\documentclass[11pt,a4paper]{article}
\usepackage[margin=28mm,headheight=14pt]{geometry}
\usepackage[T1]{fontenc}
\usepackage{lmodern,microtype}
\usepackage{amsmath,amssymb,amsthm,mathtools,mathrsfs}
\usepackage{enumitem}
\usepackage{indentfirst,etoolbox,needspace}
\usepackage{xcolor,tikz}
\usepackage{comment}
\usepackage{hyperref}
\hypersetup{colorlinks=true,linkcolor=blue!35!black,citecolor=green!30!black,
 urlcolor=blue!45!black,
 pdftitle={Positivity of supercritical LYZ pairs},
 pdfauthor={Bowen Liu, Mao Sheng, Shing-Tung Yau},
 pdfsubject={Supercritical Leung--Yau--Zaslow equation}}
\allowdisplaybreaks[2]
\numberwithin{equation}{section}
\makeatletter
\renewcommand\paragraph{\@startsection{paragraph}{4}{\z@}%
 {2.5ex \@plus 1ex \@minus .2ex}{1ex \@plus .2ex}%
 {\normalfont\normalsize\bfseries}}
\makeatother
\newtheorem{theorem}{Theorem}[section]
\newtheorem*{theorem*}{Theorem}
\newtheorem{proposition}[theorem]{Proposition}
\newtheorem{lemma}[theorem]{Lemma}
\newtheorem{corollary}[theorem]{Corollary}
\newtheorem{conjecture}[theorem]{Conjecture}
\theoremstyle{definition}
\newtheorem{definition}[theorem]{Definition}
\newtheorem{example}[theorem]{Example}
\newtheorem{remark}[theorem]{Remark}
\newcommand{\R}{\mathbb R}
\newcommand{\C}{\mathbb C}
\newcommand{\croot}{C_{\mathrm{LYZ}}}
\newcommand{\Rpair}[2]{\mathcal R_{#1,#2}}
\newcommand{\OX}{\mathcal O_X}
\DeclareMathOperator{\Amp}{Amp}
\DeclareMathOperator{\Nef}{Nef}
\DeclareMathOperator{\Psef}{Psef}
\DeclareMathOperator{\Cone}{Cone}
\DeclareMathOperator{\PL}{PL}
\DeclareMathOperator{\arccot}{arccot}

\title{Positivity of supercritical LYZ pairs}
\author{Bowen Liu\thanks{\raggedright Yau Mathematical Sciences Center, Tsinghua University,
100084 Beijing, China. Email: \href{mailto:liubw22@mails.tsinghua.edu.cn}{\texttt{liubw22@mails.tsinghua.edu.cn}}.}
\and Mao Sheng\thanks{\raggedright Yau Mathematical Sciences Center, Tsinghua University,
100084 Beijing, China. Yanqi Lake Beijing Institute of Mathematical Sciences
and Applications, 101408 Beijing, China.
Email: \href{mailto:msheng@tsinghua.edu.cn}{\texttt{msheng@tsinghua.edu.cn}}.}
\and Shing-Tung Yau\thanks{\raggedright Yau Mathematical Sciences Center, Tsinghua University,
100084 Beijing, China. Email: \href{mailto:styau@tsinghua.edu.cn}{\texttt{styau@tsinghua.edu.cn}}.}}
\date{}

\begin{document}
\maketitle

\begin{abstract}
We introduce the roots of supercritical
Leung--Yau--Zaslow pairs and prove pointwise strong positivity of the
associated Schur forms.
The resulting LYZ cones
contain the pliant cone on smooth projective threefolds, with
strict inclusion on a Fujino--Sato threefold. We also establish
Hard Lefschetz and Hodge--Riemann relations for the associated
Chern classes on compact K\"ahler manifolds in bidegrees
$1\le p+q\le\dim X$ with $\min\{p,q\}\le1$.
\end{abstract}

\begingroup
\small
\setlength{\parskip}{0pt}
\tableofcontents
\endgroup
\clearpage

\section*{Introduction}
The Leung--Yau--Zaslow (LYZ) equation arises as the mirror counterpart
of the special Lagrangian equation. In the semi-flat setting of
dual torus fibrations \cite{SYZ}, Leung--Yau--Zaslow \cite{LYZ}
showed that the Fourier--Mukai transform translates the special
Lagrangian condition into an equation for a line-bundle connection,
which we call the LYZ equation.

Let $(X,\omega)$ be a compact K\"ahler manifold of dimension $n\ge2$, and let $\alpha$ be a smooth closed real
$(1,1)$-form. In the supercritical phase, the equation takes the form
\[
 \sum_{i=1}^n\arccot\lambda_i(x)=\theta\in(0,\pi),
 \qquad x\in X,
\]
where $\lambda_i(x)$ are the eigenvalues of $\alpha$ with respect to
$\omega$ and $\arccot\colon\R\to(0,\pi)$. We call such a solution
$(\alpha,\omega)$ a supercritical LYZ pair. A pair of cohomology
classes $([\alpha],[\omega])$ is supercritical if it admits such
representatives.

For the solvability of the LYZ equation,
Collins--Jacob--Yau \cite{CJY} conjectured a Nakai--Moishezon type
criterion in the supercritical phase branch and proved it
for surfaces.
Gao Chen \cite{Chen} established a criterion involving uniform
stability along test families. Building on this work,
Chu--Lee--Takahashi \cite[Corollary~1.5]{CLT} proved the conjectured
criterion on projective manifolds in the supercritical phase.
The resulting inequalities express positivity of a sequence of
homogeneous polynomials in $[\alpha]$ and $[\omega]$ on proper
subvarieties.

In this paper, we organize these homogeneous polynomials appearing in
Nakai--Moishezon type criteria as the Chern classes of a collection of
real $(1,1)$-classes. Schur positivity and Hodge--Riemann properties
of these classes provide a new perspective on the positivity
implied by supercriticality.

More precisely, for a cohomological supercritical LYZ pair $([\alpha],[\omega])$
of phase $\theta$, we define its collection of \emph{LYZ root classes} by
\[
 \Rpair{[\alpha]}{[\omega]}:=\{[D_0],\ldots,[D_{n-2}]\},
 \qquad
 [D_j]:=[\alpha]-\cot\frac{\theta+j\pi}{n-1}\,[\omega].
\]
Elementary symmetric and Schur polynomials in these root classes define
classes $c_m(\Rpair{[\alpha]}{[\omega]})$ and $s_\nu(\Rpair{[\alpha]}{[\omega]})$. Up to positive
normalizing constants, the Chern classes are exactly the classes
appearing in the Nakai--Moishezon type criteria
(Corollary~\ref{cor:root-chern-positive}). 

For representatives
$(\alpha,\omega)$ solving the LYZ equation, these Chern forms $c_m(\alpha,\omega)$ are pointwise strongly positive for $1\le m\le n-1$, as follows from \cite[Lemma~8.2]{CJY}. Our first result shows that the Schur forms $s_\nu(\alpha,\omega)$ are also pointwise strongly
positive (Theorem~\ref{thm:root-schur-positive}).  

Fulger--Lehmann \cite{FL} introduced the pliant cone, which is generated by products of Schur classes of
globally generated vector bundles. In
\cite[Sections~6.4--6.5]{FL}, they ask how the resulting cones
compare, particularly for curves, and how other constructions of
positive classes fit into this picture.

Following this approach, we use supercritical LYZ pairs to construct such a cone on every
compact K\"ahler manifold. Let $\croot^m(X)$ be the closed convex
cone generated by products of Schur classes of supercritical LYZ pairs, of total degree $m$.
Theorem~\ref{thm:analytic-cone-inclusions} gives
\[
 S^m\overline{\mathcal K(X)}\subseteq\croot^m(X)
 \subseteq\bigl(\overline{\mathrm{Eff}}_{\mathrm{an}}^{n-m}(X)\bigr)^\vee,
 \qquad 1\le m\le n.
\]
The left-hand cone is generated by products of nef $(1,1)$-classes, and the right-hand cone consists of classes pairing nonnegatively with
every $m$-dimensional analytic subvariety. 

On smooth projective threefolds, our second main result gives a direct
comparison with the pliant cone. To be explicit, we prove
\[
 S^2\Nef^1(X)\subseteq\PL^2(X)
 \subseteq\croot^2(X)\cap N^2(X)_\R
 \subseteq\Nef^2(X)
\]
(Theorem~\ref{thm:pliant-inclusion}). Moreover, we prove that the inclusion of the pliant
cone in the LYZ cone is strict on the smooth projective toric
threefold of Fujino--Sato \cite{FS}. In particular, this gives a
threefold example of the separation between pliant and nef curve
classes considered in \cite[Section~6.4]{FL}.

The proof of the cone inclusion uses the Hodge--Riemann property
of Chern classes of ample vector bundles
\cite[Proposition~3.8]{PRT}. This connection suggests asking whether
the Chern classes of LYZ root classes satisfy analogous Hodge--Riemann relations
or, more generally, a K\"ahler package comprising Hard Lefschetz,
Hodge--Riemann relations, and Lefschetz decomposition.
In this direction, Fu--Yau--Zhang--Zhang
\cite[Proposition~4.1 and Theorem~4.3]{FYZZ} also obtained related
Hodge-index inequalities on $H^{1,1}(X,\R)$, including boundary
cases of the supercritical LYZ equation.

Our third main result proves Hard Lefschetz and the Hodge--Riemann
relations for the Chern classes of LYZ root classes on compact K\"ahler manifolds in
bidegrees $1\le p+q\le n$ with $\min\{p,q\}\le1$
(Theorem~\ref{sp:thm:global}). As a consequence, the K\"ahler package
holds in dimensions three and four
(Corollary~\ref{cor:low-dimensional-package}).
We conjecture that the K\"ahler package holds in arbitrary dimension
(Conjecture~\ref{sp:conj:full-package}).

The paper is organized as follows. Section~\ref{sec:roots}
introduces LYZ root classes and their characteristic classes.
Section~\ref{sec:pliant} proves Schur positivity and compares the
associated cones. Section~\ref{sp:sec:small-phase} establishes
the Hodge--Riemann relations.
Appendix~\ref{app:FS-data} gives the toric intersection data.

\paragraph*{Acknowledgments}
The second named author would like to thank the Chinese Academy of Sciences
for financial support through the Project for Young Scientists in Basic
Research (Grant No.~YSBR-032).

\paragraph*{AI disclosure}
The construction of the LYZ root classes and the ideas of defining the
LYZ cones and studying the Hodge--Riemann properties of supercritical LYZ pairs
originated with the authors. The strict-inclusion example
on the Fujino--Sato threefold was found by GPT-5.6 Sol Pro. The authors
also used AI tools to improve the exposition, and take full responsibility for the paper.

\clearpage
\section{LYZ root classes from supercritical LYZ pairs}\label{sec:roots}
Throughout this section, let $X$ be a compact K\"ahler manifold of
dimension $n\ge2$. We write
$H^{1,1}(X,\R):=H^{1,1}(X)\cap H^2(X,\R)$ for the space of
real $(1,1)$-cohomology classes and denote the K\"ahler cone by $\mathcal K(X)$.

\subsection{The supercritical LYZ equation and its root polynomials}

Let $\alpha$ be a smooth closed real $(1,1)$-form and let $\omega$
be a K\"ahler form on $X$. At each point $x\in X$, let
$\lambda_1(x),\ldots,\lambda_n(x)$ be the eigenvalues of $\alpha$
with respect to $\omega$, that is,
\[
 \omega=\sqrt{-1}\sum_{i=1}^n\mathrm{d}z^i\wedge\mathrm{d}\bar z^i,
 \qquad
 \alpha=\sqrt{-1}\sum_{i=1}^n\lambda_i\,\mathrm{d}z^i\wedge\mathrm{d}\bar z^i.
\]
When the central charge $\int_X([\alpha]+\sqrt{-1}[\omega])^n\ne0$, the
\emph{topological phase} $\theta\in\R/(2\pi\mathbb Z)$ is determined by
\[
 e^{-\sqrt{-1}\theta}\int_X([\alpha]+\sqrt{-1}[\omega])^n
 \in\R_{>0}.
\]
A \emph{real lift} of $\theta$ is a number $\Theta\in\R$ such that
$\theta=\Theta\bmod 2\pi$. For such a lift, the
Leung--Yau--Zaslow (LYZ) equation is
\[
 \sum_{i=1}^n\arccot\lambda_i(x)=\Theta
 \qquad\text{for every }x\in X,
\]
where $\arccot\colon\R\to(0,\pi)$.

\begin{definition}
\leavevmode
\begin{enumerate}[label=(\roman*)]
\item A solution $(\alpha,\omega)$ of the LYZ equation is called a
\emph{supercritical LYZ pair} if $\Theta\in(0,\pi)$.
In this case, we denote the unique lift in $(0,\pi)$ again by $\theta$.
\item A \emph{cohomological supercritical LYZ pair} is a pair
\[
 ([\alpha],[\omega])\in H^{1,1}(X,\R)\times\mathcal K(X)
\]
that admits representatives forming a supercritical LYZ pair.
\end{enumerate}
\end{definition}

For fixed classes $[\alpha]$ and $[\omega]$, we choose a K\"ahler
representative $\omega$ and an initial representative $\alpha_0$ of
$[\alpha]$. Solving the LYZ equation with respect to $\omega$
means finding a smooth real function $\varphi$ such that
\[
 \alpha_\varphi:=\alpha_0+\sqrt{-1}\partial\bar\partial\varphi
\]
forms a supercritical LYZ pair with $\omega$. Chu--Lee--Takahashi \cite[Corollary~1.5]{CLT} proved the following
Nakai--Moishezon type criterion.

\begin{theorem}\label{thm:LYZ-criterion}

Let \(X\) be a smooth projective variety of dimension \(n\), and let
\[
 ([\alpha],[\omega])\in H^{1,1}(X,\R)\times\mathcal K(X),
\]
with $\theta\in(0,\pi)$ satisfying
\begin{equation}\label{eq:LYZ}
\int_X
\left(
\operatorname{Re}([\alpha]+\sqrt{-1}[\omega])^n
-\cot\theta\,
\operatorname{Im}([\alpha]+\sqrt{-1}[\omega])^n
\right)
=0.
\end{equation}
Then the pair $([\alpha],[\omega])$ is a cohomological supercritical LYZ pair
if and only if, for
every irreducible proper subvariety \(Y\subset X\) of dimension
\(1\le m\le n-1\),
\[
\int_Y
\left(
\operatorname{Re}([\alpha]+\sqrt{-1}[\omega])^m
-\cot\theta\,
\operatorname{Im}([\alpha]+\sqrt{-1}[\omega])^m
\right)
>0.
\]
\end{theorem}

The numerical inequalities in this criterion involve homogeneous
polynomials in $[\alpha]$ and $[\omega]$. We study these polynomials
by first fixing a pair
$([\alpha],[\omega])\in H^{1,1}(X,\R)\times\mathcal K(X)$
and $\theta\in(0,\pi)$. For $0\le m\le n$, we define
the real $(m,m)$-cohomology classes
\begin{equation}
 [P_m]:=\operatorname{Re}([\alpha]+\sqrt{-1}[\omega])^m
 -\cot\theta\,\operatorname{Im}([\alpha]+\sqrt{-1}[\omega])^m.
\end{equation}
We set
\[
 s:=\cot\theta,\qquad [\kappa]:=[\alpha]-s[\omega].
\]

\begin{lemma}\label{lem:Pm-shifted-identity}
For every $0\le m\le n$,
\[
 [P_m]=\operatorname{Im}\left(
 (s+\sqrt{-1})([\kappa]+(s-\sqrt{-1})[\omega])^m\right).
\]
\end{lemma}
\begin{proof}
Since $[\kappa]+(s-\sqrt{-1})[\omega]=[\alpha]-\sqrt{-1}[\omega]$ and
$[\alpha],[\omega]$ are real, we have
\[
\begin{aligned}
 \operatorname{Im}\left(
 (s+\sqrt{-1})([\kappa]+(s-\sqrt{-1})[\omega])^m\right)
 &=\operatorname{Im}\left(
 (s+\sqrt{-1})\overline{([\alpha]+\sqrt{-1}[\omega])^m}\right)\\
 &=\operatorname{Re}([\alpha]+\sqrt{-1}[\omega])^m
   -s\operatorname{Im}([\alpha]+\sqrt{-1}[\omega])^m\\
 &=[P_m].
\end{aligned}
\]
\end{proof}

With the same parameter $s=\cot\theta$, we define the homogeneous polynomials
\begin{equation}\label{eq:poly}
 P_m(x,y):=\operatorname{Im}\left(
 (s+\sqrt{-1})(x+(s-\sqrt{-1})y)^m\right),\qquad m\ge0.
\end{equation}
Lemma~\ref{lem:Pm-shifted-identity} identifies the cohomology classes above as
$[P_m]=P_m([\kappa],[\omega])$.
In particular, \eqref{eq:LYZ} is the equality $\int_X[P_n]=0$.
We set $p_m(x):=P_m(x,1)$. Differentiating \eqref{eq:poly} gives
\[
 \frac{\partial}{\partial x}P_m(x,y)=mP_{m-1}(x,y).
\]
Setting $y=1$, we obtain
\begin{equation}\label{eq:appell}
 \frac{\mathrm{d}}{\mathrm{d}x}p_m(x)=m p_{m-1}(x),\qquad m\ge1.
\end{equation}
\begin{example}\label{ex:first-root-polynomials}
The first four polynomials are
\[
\begin{aligned}
 p_0(x)&=1, & p_1(x)&=x,\\
 p_2(x)&=x^2-(s^2+1), & p_3(x)&=x^3-3(s^2+1)x-2s(s^2+1).
\end{aligned}
\]
\end{example}

\begin{lemma}
For every \(m\ge1\), the polynomial \(p_m\) has \(m\) simple real
roots, given by
\begin{equation}\label{eq:roots}
\mu_{m,j}
:=
\cot\left(\frac{\theta+j\pi}{m}\right)-s,
\qquad
0\le j\le m-1.
\end{equation}
For $m\ge2$, between each pair of consecutive roots of $p_m$ lies
exactly one root of $p_{m-1}$. Explicitly,
\[
 \mu_{m,j}>\mu_{m-1,j}>\mu_{m,j+1},
 \qquad 0\le j\le m-2.
\]
\end{lemma}

\begin{proof}
For $0<t<\pi$, we substitute $x=\cot t-s$ into \eqref{eq:poly}.
Using $s=\cot\theta$, we obtain
\[
\begin{aligned}
 p_m(\cot t-s)
 &=\operatorname{Im}\bigl((\cot\theta+\sqrt{-1})
                          (\cot t-\sqrt{-1})^m\bigr)\\
 &=\operatorname{Im}\left(
   \frac{\cos\theta+\sqrt{-1}\sin\theta}{\sin\theta}
   \frac{(\cos t-\sqrt{-1}\sin t)^m}{\sin^m t}\right)\\
 &=\frac{\operatorname{Im}e^{\sqrt{-1}(\theta-mt)}}
          {\sin\theta\,\sin^m t}\\
 &=\frac{\sin(\theta-mt)}{\sin\theta\,\sin^m t}.
\end{aligned}
\]
Taking $t=(\theta+j\pi)/m$, $0\le j\le m-1$, gives the $m$
distinct roots in \eqref{eq:roots}. These are all the roots,
since $p_m$ is monic of degree $m$.
Rolle's theorem and \eqref{eq:appell} place exactly one zero of
$p_{m-1}$ between consecutive zeros of $p_m$.
\end{proof}

\subsection{LYZ root classes and their characteristic classes}

The roots of $p_{n-1}$ give the following cohomological construction.

\begin{definition}\label{def:formal-root-bundle}
Let $n\ge2$ and let $\mathcal R:=\{[D_0],\ldots,[D_{n-2}]\}$
be a collection of classes in $H^{1,1}(X,\R)$.
We call $\mathcal R$ a \emph{collection of LYZ root classes} if the classes
$[D_j]$ lie on an affine line with K\"ahler direction
$[\omega]\in\mathcal K(X)$, and there exist
$[\alpha]\in H^{1,1}(X,\R)$ and $\theta\in(0,\pi)$ satisfying
\eqref{eq:LYZ} such that the classes can be ordered as
\begin{equation}\label{eq:abstract-root-positions}
 [D_j]:=[\alpha]-\cot\frac{\theta+j\pi}{n-1}\,[\omega],
 \qquad 0\le j\le n-2.
\end{equation}
If $([\alpha],[\omega])$ can be chosen to be a cohomological supercritical LYZ pair, we call $\mathcal R$ \emph{supercritical}.
\end{definition}

The pair $([\alpha],[\omega])$ and the phase $\theta$ in
Definition~\ref{def:formal-root-bundle} are not a priori determined by
$\mathcal R$. On a surface, we fix a class $[\kappa]$ with
$[\kappa]^2>0$ and a K\"ahler ray.
For each $\theta\in(0,\pi)$, we choose $[\omega]$ on this ray so that
$[\omega]^2=\sin^2\theta\,[\kappa]^2$, and set
$[\alpha]:=[\kappa]+\cot\theta\,[\omega]$.
Then equation \eqref{eq:LYZ} holds and $\mathcal R=\{[\kappa]\}$.
Varying $\theta$ gives infinitely many distinct pairs with the same
root class $[\kappa]$.

\begin{proposition}
On a compact K\"ahler surface $X$, the space of cohomological supercritical LYZ pairs is naturally identified with
\[
 \left\{([\kappa],[\omega],s)\in
 \mathcal K(X)\times\mathcal K(X)\times\R
 \;:\;[\kappa]^2=(1+s^2)[\omega]^2\right\}.
\]
\end{proposition}
\begin{proof}
Given a cohomological supercritical LYZ pair $([\alpha],[\omega])$ with phase
$\theta\in(0,\pi)$ satisfying \eqref{eq:LYZ}, its central charge $\int_X([\alpha]+\sqrt{-1}[\omega])^2$
has positive imaginary part $2[\alpha][\omega]$, so the pair determines
\[
 s:=\frac{[\alpha]^2-[\omega]^2}{2[\alpha][\omega]}=\cot\theta,
 \qquad [\kappa]:=[\alpha]-s[\omega].
\]
By the solvability criterion on surfaces
\cite[Proposition~8.5]{CJY}, the class $[\kappa]$ is K\"ahler, and
equation \eqref{eq:LYZ} gives $[\kappa]^2=(1+s^2)[\omega]^2$.

Conversely, given such a triple, we set $[\alpha]:=[\kappa]+s[\omega]$
and $\theta:=\arccot s$. The Hodge index theorem gives
\[
 [\alpha][\omega]=[\kappa][\omega]+s[\omega]^2
 \ge\bigl(\sqrt{1+s^2}+s\bigr)[\omega]^2>0.
\]
Since $[\alpha]^2-[\omega]^2=2s[\alpha][\omega]$, we have
\[
 \int_X([\alpha]+\sqrt{-1}[\omega])^2
 =2[\alpha][\omega](s+\sqrt{-1}),
\]
whose argument is $\theta$.
The class $[\alpha]-\cot\theta\,[\omega]=[\kappa]$ is K\"ahler, so the
solvability criterion shows that $([\alpha],[\omega])$ is a cohomological
supercritical LYZ pair. The two constructions are inverse.
This completes the proof.
\end{proof}

In contrast, for $n\ge3$, the collection of LYZ root classes determines the triple uniquely.

\begin{proposition}
Let $X$ be a compact K\"ahler manifold of dimension $n\ge3$.
There is a natural bijection between triples
$([\alpha],[\omega],\theta)\in H^{1,1}(X,\R)\times\mathcal K(X)\times(0,\pi)$,
with $\theta$ satisfying \eqref{eq:LYZ}, and collections of LYZ root classes on $X$.
\end{proposition}
\begin{proof}
The assignment \eqref{eq:abstract-root-positions} defines the map
from triples to collections of LYZ root classes. It is surjective by
Definition~\ref{def:formal-root-bundle}.

To prove injectivity, suppose that two triples
$([\alpha],[\omega],\theta)$ and
$([\widetilde\alpha],[\widetilde\omega],\widetilde\theta)$ give the same
collection $\mathcal R=\{[D_0],[D_1],\ldots,[D_{n-2}]\}$. We set
\[
\begin{aligned}
 [\kappa]&:=[\alpha]-\cot\theta\,[\omega],
 & [\Omega]&:=\frac{[\omega]}{\sin\theta},\\
 [\widetilde\kappa]&:=[\widetilde\alpha]
   -\cot\widetilde\theta\,[\widetilde\omega],
 & [\widetilde\Omega]&:=\frac{[\widetilde\omega]}{\sin\widetilde\theta}.
\end{aligned}
\]
Both $[\Omega]$ and $[\widetilde\Omega]$ are K\"ahler.
The collection $\mathcal R$ defines
$Q(t):=\prod_{j=0}^{n-2}(t+[D_j])$.
For the first triple, the root factorization and the derivative identity
\eqref{eq:appell} give
\[
 Q(t)=P_{n-1}(t+[\kappa],[\omega]),
 \qquad
 [t^{n-1-m}]Q(t)=\binom{n-1}{m}P_m([\kappa],[\omega]),
\]
where $0\le m\le n-1$ and $[t^k]$ denotes the coefficient of $t^k$.
The same formulas hold for the second triple, with its phase
$\widetilde\theta$. For every $n\ge3$, the formulas in
Example~\ref{ex:first-root-polynomials} show that comparing the coefficients
of $t^{n-2}$ and $t^{n-3}$ for the two triples gives
\[
 [\kappa]=[\widetilde\kappa],
 \qquad
 [\kappa]^2-[\Omega]^2
 =[\widetilde\kappa]^2-[\widetilde\Omega]^2.
\]
Thus $[\Omega]^2=[\widetilde\Omega]^2$.
We set $H:=[\Omega]+[\widetilde\Omega]\in\mathcal K(X)$. Then
\[
 ([\Omega]-[\widetilde\Omega])H^{n-2}
 =([\Omega]^2-[\widetilde\Omega]^2)H^{n-3}=0,
\]
and thus the Hard Lefschetz theorem for $H$ gives $[\Omega]=[\widetilde\Omega]$.

It remains to prove $\theta=\widetilde\theta$.
Suppose first that $n=3$. Since $p_2(x)=x^2-\csc^2\theta$, the root classes are
\begin{equation}\label{eq:threefold-root-classes}
 [D_0]=[\kappa]-[\Omega],
 \qquad
 [D_1]=[\kappa]+[\Omega].
\end{equation}
Using the common classes $[\kappa]$ and $[\Omega]$, the phase equation
\eqref{eq:LYZ} for the two triples becomes
\begin{equation}\label{eq:recover-costheta}
 0=[\kappa]^3-3[\kappa][\Omega]^2
   -2\cos\theta\,[\Omega]^3
  =[\kappa]^3-3[\kappa][\Omega]^2
   -2\cos\widetilde\theta\,[\Omega]^3.
\end{equation}
Subtracting and using $[\Omega]^3>0$ gives
$\cos\theta=\cos\widetilde\theta$.

Now suppose that $n>3$. Comparing the coefficients of $t^{n-4}$ in
$Q(t)$ and cancelling the nonzero factor $\binom{n-1}{3}$ gives
\[
 [\kappa]^3-3[\kappa][\Omega]^2-2\cos\theta\,[\Omega]^3
 =[\kappa]^3-3[\kappa][\Omega]^2-2\cos\widetilde\theta\,[\Omega]^3.
\]
The class $[\Omega]^3$ is nonzero because $\int_X[\Omega]^n>0$.
Subtracting therefore gives $\cos\theta=\cos\widetilde\theta$ in this case as well.
Since $\theta,\widetilde\theta\in(0,\pi)$, we have
$\theta=\widetilde\theta$. Consequently,
\[
 [\omega]=\sin\theta\,[\Omega]=[\widetilde\omega],
 \qquad
 [\alpha]=[\kappa]+\cos\theta\,[\Omega]=[\widetilde\alpha].
\]
The two triples are equal. This proves injectivity.
\end{proof}

For a collection of LYZ root classes $\mathcal R=\{[D_0],\ldots,[D_{n-2}]\}$, we define its
\emph{Chern and Schur classes} by
\[
 c_m(\mathcal R):=e_m([D_0],\ldots,[D_{n-2}]),
 \qquad
 s_\nu(\mathcal R):=s_\nu([D_0],\ldots,[D_{n-2}]),
\]
where $e_m$ is the $m$-th elementary symmetric polynomial and $s_\nu$
is the Schur polynomial associated with a partition $\nu$.
We set
$c_0(\mathcal R):=1$ and $c_m(\mathcal R):=0$ for $m<0$ or $m>n-1$.

\begin{proposition}\label{prop:chern}
Let $\mathcal R$ be the collection of LYZ root classes defined by a triple
$([\alpha],[\omega],\theta)$. We set
$[\kappa]:=[\alpha]-\cot\theta\,[\omega]$. For $0\le m\le n-1$,
\begin{equation}\label{eq:chern}
 c_m(\mathcal R)=\binom{n-1}{m}P_m([\kappa],[\omega]).
\end{equation}
\end{proposition}
\begin{proof}
By \eqref{eq:roots} and \eqref{eq:abstract-root-positions}, we have
$P_{n-1}(x,y)=\prod_{j=0}^{n-2}(x-\mu_{n-1,j}y)$ and
$[D_j]=[\kappa]-\mu_{n-1,j}[\omega]$.
Repeated differentiation of \eqref{eq:poly} gives
$\partial_x^{\,n-1-m}P_{n-1}(x,y)
=\frac{(n-1)!}{m!}P_m(x,y)$ for $0\le m\le n-1$.
Thus, for a formal variable $t$, Taylor's formula gives
\[
\begin{aligned}
 \prod_{j=0}^{n-2}(t+[D_j])
 &=P_{n-1}([\kappa]+t,[\omega])\\
 &=\sum_{m=0}^{n-1}
   \frac{\bigl(\partial_x^{\,n-1-m}P_{n-1}\bigr)([\kappa],[\omega])}
        {(n-1-m)!}\,t^{n-1-m}\\
 &=\sum_{m=0}^{n-1}\binom{n-1}{m}
   P_m([\kappa],[\omega])t^{n-1-m}.
\end{aligned}
\]
Comparing the coefficients of $t^{n-1-m}$ gives \eqref{eq:chern}.
\end{proof}

\begin{corollary}\label{cor:root-chern-positive}
Let $X$ be a smooth projective variety, and let
$([\alpha],[\omega])\in H^{1,1}(X,\R)\times\mathcal K(X)$,
with $\theta\in(0,\pi)$ satisfying \eqref{eq:LYZ}.
Let $\mathcal R$ be the collection defined by
\eqref{eq:abstract-root-positions} for this triple.
Then the pair $([\alpha],[\omega])$ is a cohomological supercritical LYZ pair
if and only if $\mathcal R$ is Chern positive, that is,
\[
 \int_Yc_m(\mathcal R)>0
\]
for every irreducible proper subvariety $Y\subset X$ of dimension
$m$, with $1\le m\le n-1$.
\end{corollary}
\begin{proof}
This follows from Theorem~\ref{thm:LYZ-criterion} and
\eqref{eq:chern}.
\end{proof}

From now on, for a cohomological supercritical LYZ pair
$([\alpha],[\omega])$, we always choose smooth closed representatives
$(\alpha,\omega)$ solving the LYZ equation with phase
$\theta\in(0,\pi)$. We write $\Rpair{[\alpha]}{[\omega]}$ for the collection of LYZ root
classes and $\Rpair{\alpha}{\omega}:=\{D_0,\ldots,D_{n-2}\}$ for the
collection of LYZ root forms, where we define
\[
 D_j:=\alpha-\cot\frac{\theta+j\pi}{n-1}\,\omega,
 \qquad 0\le j\le n-2,
\]
and their \emph{Chern and Schur forms} by
\[
 c_m(\alpha,\omega):=e_m(D_0,\ldots,D_{n-2}),
 \qquad
 s_\nu(\alpha,\omega):=s_\nu(D_0,\ldots,D_{n-2}).
\]
\section{Positive cones from supercritical LYZ pairs}\label{sec:pliant}
Let $X$ be a compact K\"ahler manifold of dimension $n\ge2$.
We write
\[
 H^{m,m}(X,\R):=H^{m,m}(X)\cap H^{2m}(X,\R)
\]
for the space of real $(m,m)$-cohomology classes.
Motivated by the mixed Schur products defining the pliant cone
\cite[Definition~3.2]{FL}, we consider the corresponding cone of
Schur classes of collections of LYZ root classes in cohomology.

\begin{definition}
For $1\le m\le n$, the $m$-th \emph{LYZ cone} in
$H^{m,m}(X,\R)$ is
\begin{equation}
 \croot^m(X):=\overline{\Cone\!\left\{
 \prod_{i=1}^r s_{\nu^{(i)}}(\Rpair{[\alpha_i]}{[\omega_i]})
 \;\middle|\;
 \begin{gathered}
 ([\alpha_i],[\omega_i])\text{ are cohomological}\\
 \text{supercritical LYZ pairs on }X,\\
 r\ge1,\quad |\nu^{(i)}|\ge1,\\
 \ell(\nu^{(i)})\le n-1,\quad
 \sum_{i=1}^r|\nu^{(i)}|=m
 \end{gathered}\right\}}.
\end{equation}
Here each $\nu^{(i)}$ is a partition. We write $|\nu|$ for the sum of
the parts of $\nu$ and $\ell(\nu)$ for its number of nonzero parts.
\end{definition}

In this section, we first establish pointwise strong positivity of the Schur forms and compare
the LYZ cone with products of nef $(1,1)$-classes and the dual cone
of effective analytic cycle classes. We then study its algebraic part on
projective threefolds, proving that it contains the pliant cone and
that this inclusion is strict on the Fujino--Sato example.

\subsection{Pointwise strong positivity of Schur forms}

Let $(\alpha,\omega)$ be a supercritical LYZ pair with phase
$\theta\in(0,\pi)$, and let $\Rpair{\alpha}{\omega}=\{D_0,\dots,D_{n-2}\}$ be its
collection of LYZ root forms, where 
\[
 D_j=\alpha-q_j\omega,\qquad
 q_j:=\cot\frac{\theta+j\pi}{n-1},\qquad 0\le j\le n-2.
\]
For a partition $\nu=(\nu_1\ge\cdots\ge\nu_\ell>0)$,
let $\nu'$ be its conjugate.
The \emph{Schur forms} are defined by the determinant formula
\[
 s_\nu(\alpha,\omega)
 :=\det\bigl(c_{\nu'_i-i+j}(\alpha,\omega)\bigr)_{1\le i,j\le\nu_1}
 =s_\nu(D_0,\ldots,D_{n-2}).
\]
We set
\[
 F_\nu(x,y):=s_\nu(x-q_0y,\ldots,x-q_{n-2}y).
\]
Then
\begin{equation}\label{eq:schur-root-F}
 s_\nu(\alpha,\omega)=F_\nu(\alpha,\omega).
\end{equation}

At a point $p\in X$, we choose an $\omega$-orthonormal complex
coframe $\mathrm{d}z^1,\ldots,\mathrm{d}z^n$ diagonalizing $\alpha$.
Then
\[
 \omega=\sqrt{-1}\sum_{i=1}^n\mathrm{d}z^i\wedge\mathrm{d}\bar z^i,
 \qquad
 \alpha=\sqrt{-1}\sum_{i=1}^n\lambda_i\,\mathrm{d}z^i\wedge\mathrm{d}\bar z^i,
\]
where $\lambda_i=\cot\phi_i$, $\phi_i>0$, and
$\sum_{i=1}^n\phi_i=\theta$.

Let $\nu\vdash m$ with $\ell(\nu)\le n-1$. For angles
$0<\psi_1,\ldots,\psi_m<\pi$, we define
\begin{equation}\label{eq:Clambda}
C_\nu(\psi_1,\ldots,\psi_m)
:=
[t_1\cdots t_m]\,
F_\nu\left(
\sum_{i=1}^m t_i\cot\psi_i,
\sum_{i=1}^m t_i
\right).
\end{equation}
Here the notation $[t_1\cdots t_m]$ denotes the coefficient of $t_1\cdots t_m$.
In this coframe, \eqref{eq:schur-root-F} gives
\begin{equation}\label{eq:schur-squarefree-expansion}
s_\nu(\alpha,\omega)
=
(\sqrt{-1})^{\,m^2}
\sum_{\substack{I\subset\{1,\ldots,n\}\\|I|=m}}
C_\nu((\phi_i)_{i\in I})\,\mathrm{d}z^I\wedge\mathrm{d}\bar z^I.
\end{equation}
Here $\mathrm{d}z^I:=\bigwedge_{i\in I}\mathrm{d}z^i$ and
$\mathrm{d}\bar z^I:=\overline{\mathrm{d}z^I}$, with indices in increasing order.

To prove pointwise strong positivity of the Schur forms associated with
$\Rpair{\alpha}{\omega}$,
it suffices to show that every coefficient in
\eqref{eq:schur-squarefree-expansion} is nonnegative.
We fix $J\subseteq\{1,\ldots,n\}$ with $|J|=m$ and consider
$C_\nu((\phi_i)_{i\in J})$.
We first prove nonnegativity when these $m$ selected angles are
all equal. Then we use averaging to compare an arbitrary selected
angle tuple with the equal-angle tuple of the same sum.

\subsubsection{Equal-angle case}

Suppose that $\phi_i=\phi$ for every $i\in J$.
Since the selected angles have sum at most $\theta$, we have
$m\phi\le\theta$, and the coefficient to be studied is
$C_\nu(\phi,\ldots,\phi)$.

We define the $(n-1)\times(n-1)$ Hermitian matrix $M$ by
\[
 M_{ij}:=\begin{cases}
 -\cot\theta, & i=j,\\
 -\cot\theta+\sqrt{-1}, & i>j,\\
 -\cot\theta-\sqrt{-1}, & i<j,
 \end{cases}
\]
and, for $0<\phi<\pi$, we set
\[
 K_\phi:=\cos\phi\,I_{n-1}+\sin\phi\,M.
\]

For the rest of this section, we write $A[J]:=(A_{ij})_{i,j\in J}$ for the principal submatrix
indexed by $J$, with indices in increasing order.

\begin{lemma}
The eigenvalues of $M$ are $-q_0,\ldots,-q_{n-2}$, and those of
$K_\phi$ are
\[
 \xi_i(\phi):=\cos\phi-q_{i-1}\sin\phi,
 \qquad 1\le i\le n-1.
\]
For every nonempty $J\subseteq\{1,\ldots,n-1\}$, with $j=|J|$,
and $x,y\in\R$,
\begin{equation}\label{eq:schur-matrix-determinant}
 \det(xI_j+yM[J])
 =\operatorname{Re}(x+\sqrt{-1}y)^j
  -\cot\theta\,\operatorname{Im}(x+\sqrt{-1}y)^j.
\end{equation}
Thus, every principal minor of $K_\phi$ of order $j$ is given by
\begin{equation}\label{eq:K-principal-minor}
 \det K_\phi[J]=\frac{\sin(\theta-j\phi)}{\sin\theta},
 \qquad |J|=j.
\end{equation}
\end{lemma}
\begin{proof}
We write $\mathbf1:=(1,\ldots,1)^t\in\R^j$ and decompose
$xI_j+yM[J]=T+(-\cot\theta+\sqrt{-1})y\mathbf1\mathbf1^*$,
where $T$ is upper triangular, with diagonal $x-\sqrt{-1}y$ and
entries $-2\sqrt{-1}y$ above the diagonal.
By multilinearity in the columns, the determinant $\det(T+ty\mathbf1\mathbf1^*)$ is affine in $t$. Indeed, every term involving at least two factors of $t$ vanishes, since the corresponding columns are identical. 

At $t=0$ its value is $(x-\sqrt{-1}y)^j$. At $t=2\sqrt{-1}$, the matrix is lower triangular, with determinant $(x+\sqrt{-1}y)^j$. Thus,
\[
\begin{aligned}
\det(T+ty\mathbf1\mathbf1^*)
&=(x-\sqrt{-1}y)^j+\frac{t}{2\sqrt{-1}}\bigl((x+\sqrt{-1}y)^j-(x-\sqrt{-1}y)^j\bigr).
\end{aligned}
\]
Taking $t=-\cot\theta+\sqrt{-1}$, we obtain
\[
\begin{aligned}
\det(xI_j+yM[J])
&=\frac{1}{2}(1+\sqrt{-1}\cot\theta)(x+\sqrt{-1}y)^j+\frac{1}{2}(1-\sqrt{-1}\cot\theta)(x-\sqrt{-1}y)^j\\
&=\operatorname{Re}(x+\sqrt{-1}y)^j
-\cot\theta\,\operatorname{Im}(x+\sqrt{-1}y)^j,
\end{aligned}
\]
which is \eqref{eq:schur-matrix-determinant}. For $j=n-1$ and $y=1$, its zeros in
$x$ are precisely $q_0,\ldots,q_{n-2}$, which gives the eigenvalues of
$M$ and hence of $K_\phi$. Substituting $(x,y)=(\cos\phi,\sin\phi)$ gives
$\det K_\phi[J]=\cos(j\phi)-\cot\theta\sin(j\phi)$,
which is \eqref{eq:K-principal-minor}.
\end{proof}

By homogeneity and \eqref{eq:Clambda}, the eigenvalue formula gives
\begin{equation}\label{eq:F-K}
 C_\nu(\phi,\ldots,\phi)
 =m!F_\nu(\cot\phi,1)
 =\frac{m!}{\sin^m\phi}
 s_\nu\bigl(\xi_1(\phi),\ldots,\xi_{n-1}(\phi)\bigr).
\end{equation}
Moreover, \eqref{eq:K-principal-minor} is nonnegative for
$j\le\min\{m,n-1\}$, since $0\le\theta-j\phi<\pi$.
The principal-minor criterion therefore makes every principal
submatrix of $K_\phi$ of size at most $m$ positive semidefinite.
The equal-angle case now reduces to the following linear algebra fact.

\begin{lemma}\label{lem:local-schur-positivity}
Let $K$ be an $(n-1)\times(n-1)$ Hermitian matrix with eigenvalues
$\xi_1,\ldots,\xi_{n-1}$, and let $\nu\vdash m$ with
$\ell(\nu)\le n-1$. If every principal submatrix of $K$ of size
at most $\min\{m,n-1\}$ is positive semidefinite, then
$s_\nu(\xi_1,\ldots,\xi_{n-1})\ge0$.
\end{lemma}

\begin{proof}
Let $\chi^\nu$ be the irreducible character of the symmetric group
$S_m$ indexed by $\nu$. For an $m\times m$ matrix $A$, we define
its $\nu$-immanant by
\[
 \operatorname{Imm}_\nu(A)
 :=\sum_{\sigma\in S_m}\chi^\nu(\sigma)
   \prod_{i=1}^m A_{i,\sigma(i)}.
\]
For a positive definite Hermitian matrix $A$, Schur's immanant
inequality \cite[Theorem~1.1]{Haiman93} states that
\[
 \operatorname{Imm}_\nu(A)
 \ge\chi^\nu(\mathrm{id})\det A>0,
\]
where $\chi^\nu(\mathrm{id})$ is the dimension of the representation
and is therefore positive.
For $A\ge0$, applying this inequality to $A+\varepsilon I_m$ and
letting $\varepsilon\to0^+$ gives
$\operatorname{Imm}_\nu(A)\ge0$ by continuity.

For an ordered list of indices $I=(i_1,\ldots,i_m)$ with
$1\le i_1,\ldots,i_m\le n-1$, allowing repetitions, define
\[
 K_I:=(K_{i_a,i_b})_{a,b=1}^m.
\]
The immanant expansion of Schur polynomials
\cite[Proposition~2.4 and Remark~2.6]{OO} gives
\begin{equation}\label{eq:schur-equal-angles}
 m!\,s_\nu(\xi_1,\ldots,\xi_{n-1})
 =\sum_{\substack{I=(i_1,\ldots,i_m)\\
                  1\le i_1,\ldots,i_m\le n-1}}
 \operatorname{Imm}_\nu(K_I).
\end{equation}
Fix such an $I$, and list its distinct indices as
$j_1<\cdots<j_r$. Set $J=\{j_1,\ldots,j_r\}$ and define
$R\in\mathbb R^{m\times r}$ by
\[
 R_{ab}:=
 \begin{cases}
  1, & i_a=j_b,\\
  0, & i_a\ne j_b.
 \end{cases}
\]
Then $K_I=R K[J]R^*$. Since $r\le\min\{m,n-1\}$, the matrix $K[J]$ is positive
semidefinite by hypothesis, and so is $K_I$.
Hence $\operatorname{Imm}_\nu(K_I)\ge0$ for every $I$.
Every summand in \eqref{eq:schur-equal-angles} is nonnegative, proving the claim.
\end{proof}

\begin{corollary}\label{cor:schur-equal-angle}
Let $\nu\vdash m$ with $\ell(\nu)\le n-1$ and let
$0<\phi<\pi$ satisfy $m\phi\le\theta$. Then
\[
 C_\nu(\phi,\ldots,\phi)\ge0.
\]
\end{corollary}
\begin{proof}
We apply Lemma~\ref{lem:local-schur-positivity} to $K_\phi$ and use
\eqref{eq:F-K}.
\end{proof}

\subsubsection{Averaging the phase angles}

Let $\psi_1,\ldots,\psi_m>0$ have sum $S\le\theta$.
By Corollary~\ref{cor:schur-equal-angle}, $C_\nu(S/m,\ldots,S/m)$ is nonnegative.
We will prove the following inequality by averaging:
\[
 C_\nu(\psi_1,\ldots,\psi_m)
 \ge C_\nu(S/m,\ldots,S/m).
\]

\begin{definition}
The Young diagram of a partition $\nu$ has $\nu_i$ boxes in row $i$,
with rows aligned on the left. For partitions $\mu$ and $\nu$, we write
$\mu\nearrow\nu$ if the diagram of $\nu$ is obtained from that of
$\mu$ by adding one box. We denote this box by $\nu/\mu$.
If it lies in row $i$ from the top and column $j$ from the left,
both numbered from $1$, its \emph{content} is
\[
 c(\nu/\mu):=j-i.
\]
\end{definition}

\Needspace{12\baselineskip}
\begin{example}
For $\nu=(3,1)$, the partitions $\mu\nearrow\nu$ are $(2,1)$ and
$(3)$. The shaded box is the added box $\nu/\mu$, and its number
is its content.
\begin{center}
\begin{tikzpicture}[x=0.55cm,y=0.55cm,line width=0.5pt,font=\small]
 \draw (0,0) rectangle (2,1);
 \draw (1,0)--(1,1);
 \draw (0,-1) rectangle (1,0);
 \draw[->] (2.5,0.25)--(3.5,0.25);
 \fill[blue!15] (6,0) rectangle (7,1);
 \draw (4,0) rectangle (7,1);
 \draw (5,0)--(5,1) (6,0)--(6,1);
 \draw (4,-1) rectangle (5,0);
 \node at (6.5,0.5) {$2$};
 \node at (1,-1.65) {$\mu=(2,1)$};
 \node at (5.5,-1.65) {$\nu=(3,1)$};
 \begin{scope}[xshift=5.4cm]
  \draw (0,0) rectangle (3,1);
  \draw (1,0)--(1,1) (2,0)--(2,1);
  \draw[->] (3.5,0.25)--(4.5,0.25);
  \fill[blue!15] (5,-1) rectangle (6,0);
  \draw (5,0) rectangle (8,1);
  \draw (6,0)--(6,1) (7,0)--(7,1);
  \draw (5,-1) rectangle (6,0);
  \node at (5.5,-0.5) {$-1$};
  \node at (1.5,-1.65) {$\mu=(3)$};
  \node at (6.5,-1.65) {$\nu=(3,1)$};
 \end{scope}
\end{tikzpicture}
\end{center}
The added boxes lie at $(i,j)=(1,3)$ and $(2,1)$, respectively,
so their contents are $3-1=2$ and $1-2=-1$.
\end{example}

\begin{lemma}\label{lem:schur-translation}
With $F_\varnothing:=1$, for every nonempty partition $\nu$ with $\ell(\nu)\le n-1$,
\[
 \partial_xF_\nu
 =\sum_{\mu\nearrow\nu}
  \bigl(n-1+c(\nu/\mu)\bigr)F_\mu,
\]
where the coefficients
\[
 n-1+c(\nu/\mu)=n-1+\nu_i-i\ge n-i\ge1
\]
are positive, with $i$ the row of the removed box $\nu/\mu$.
\end{lemma}
\begin{proof}
We apply the diagonal-derivative formula \cite[Remark~2.4]{GKMK}
to $s_\nu(z_1,\ldots,z_{n-1})$ with $z_k:=x-q_{k-1}y$.
The chain rule gives $\partial_xF_\nu=\sum_k\partial_{z_k}s_\nu$.
\end{proof}

\begin{proposition}\label{prop:schur-averaging}
Let $m\ge2$. Suppose that for every partition $\mu\vdash m-1$
with $\ell(\mu)\le n-1$, the coefficient $C_\mu$ is nonnegative
on every positive $(m-1)$-tuple of sum at most $\theta$.
Let $\nu\vdash m$ with $\ell(\nu)\le n-1$, and let
$a,b,\psi_3,\ldots,\psi_m>0$ satisfy
$a+b+\sum_{i=3}^m\psi_i\le\theta$. If $a',b'>0$ satisfy
\[
 a'+b'=a+b,\qquad |a'-b'|\le|a-b|,
\]
then
\begin{equation}\label{eq:averaging-difference}
 C_\nu(a',b',\psi_3,\ldots,\psi_m)
 \le C_\nu(a,b,\psi_3,\ldots,\psi_m).
\end{equation}
\end{proposition}

\begin{proof}
For convenience, we put $\sigma:=a+b=a'+b'$.
With $\psi_3,\ldots,\psi_m$ fixed, the expression in
\eqref{eq:Clambda} is symmetric and affine in $\cot a,\cot b$ respectively.
Hence we can write
\[
\begin{aligned}
 C_\nu(a,b,\psi_3,\ldots,\psi_m)
 &=A+B(\cot a+\cot b)+D\cot a\cot b\\
 &=(A+D)+(B+D\cot\sigma)(\cot a+\cot b),
\end{aligned}
\]
where the second equality follows from the identity
$\cot a\cot b=1+\cot\sigma(\cot a+\cot b)$, and $A,B,D$ are independent of $a,b$. To prove the desired result, it suffices to show
$B+D\cot\sigma\ge0$ and
$\cot a'+\cot b'\le\cot a+\cot b$.

To prove the first inequality, we define
\[
 P(\lambda_1,\lambda_2)
 :=[t_1\cdots t_m]F_\nu\!\left(
 t_1\lambda_1+t_2\lambda_2+\sum_{i=3}^m t_i\cot\psi_i,
 \sum_{i=1}^m t_i\right).
\]
This is the polynomial in the first two cotangents defining $C_\nu$,
so the same constants $A,B,D$ give
\[
 P(\lambda_1,\lambda_2)
 =A+B(\lambda_1+\lambda_2)+D\lambda_1\lambda_2.
\]
Moreover, $B+D\cot\sigma$ is the coefficient of $u$ in
$P(u,\cot\sigma)$, since
\[
 P(u,\cot\sigma)=(A+B\cot\sigma)+u(B+D\cot\sigma).
\]
To compute this coefficient from the definition of $P$, set
\[
 x:=t_2\cot\sigma+\sum_{i=3}^m t_i\cot\psi_i,
 \qquad y:=\sum_{i=2}^m t_i.
\]
The coefficient of $t_1$ in $F_\nu(x+t_1u,y+t_1)$ is
$u\,\partial_xF_\nu(x,y)+\partial_yF_\nu(x,y)$.
Taking the coefficient of $t_2\cdots t_m$ then gives
\[
 \begin{aligned}
 P(u,\cot\sigma)
 &=[t_1\cdots t_m]F_\nu(x+t_1u,y+t_1)\\
 &=[t_2\cdots t_m]\bigl(
 u\,\partial_xF_\nu(x,y)+\partial_yF_\nu(x,y)\bigr).
 \end{aligned}
\]
Since $x,y$ are independent of $u$, comparing the coefficients of $u$
in these two expressions for $P(u,\cot\sigma)$ yields
\[
 B+D\cot\sigma
 =[t_2\cdots t_m]\partial_xF_\nu(x,y).
\]
Applying Lemma~\ref{lem:schur-translation} and then the definition
\eqref{eq:Clambda} of $C_\mu$, we obtain
\[
 B+D\cot\sigma
 =\sum_{\mu\nearrow\nu}
 \bigl(n-1+c(\nu/\mu)\bigr)
 C_\mu(\sigma,\psi_3,\ldots,\psi_m).
\]
Each weight is positive. The tuple $(\sigma,\psi_3,\ldots,\psi_m)$
has $m-1$ positive entries and sum
$\sigma+\sum_{i=3}^m\psi_i\le\theta$, so every $C_\mu$ in this sum
is nonnegative by the hypothesis. Hence $B+D\cot\sigma\ge0$.

Finally, since $0<\sigma<\pi$,
\[
 \cot a+\cot b
 =\frac{\sin\sigma}
 {\sin^2(\sigma/2)-\sin^2((a-b)/2)}.
\]
The same formula holds for $a',b'$. The condition
$0\le|a'-b'|\le|a-b|<\sigma<\pi$ therefore gives
$\cot a'+\cot b'\le\cot a+\cot b$.
Together with $B+D\cot\sigma\ge0$, this proves
\eqref{eq:averaging-difference}.
\end{proof}

Induction now gives the coefficient inequality for arbitrary angles.

\begin{corollary}\label{cor:schur-barycenter}
Let $\nu\vdash m$ with $\ell(\nu)\le n-1$, and let
$\psi_1,\ldots,\psi_m>0$ have sum $S\le\theta$. Then
\[
 C_\nu(\psi_1,\ldots,\psi_m)
 \ge C_\nu\left(\frac Sm,\ldots,\frac Sm\right)\ge0.
\]
\end{corollary}
\begin{proof}
We induct on $m$. For $m=1$, the first inequality is equality and
the second follows from Corollary~\ref{cor:schur-equal-angle}.
Assume the claim for all partitions of degree $m-1$.
The hypotheses of Proposition~\ref{prop:schur-averaging} are then
satisfied. We put $u:=S/m$. Unless all entries equal $u$, we choose
$a<u<b$ among the entries and replace
\[
 (a,b)\longmapsto(u,a+b-u).
\]
At each step, one of the selected angles is replaced by $u$,
and $C_\nu$ does not increase. After finitely many steps, we obtain
\[
 C_\nu(\psi_1,\ldots,\psi_m)
 \ge C_\nu(S/m,\ldots,S/m)\ge0,
\]
where the last inequality is Corollary~\ref{cor:schur-equal-angle}.
This completes the induction.
\end{proof}

\begin{theorem}\label{thm:root-schur-positive}
Let $(\alpha,\omega)$ be a supercritical LYZ pair on a compact
K\"ahler $n$-fold $X$, with $n\ge2$.
For every partition $\nu$ with $1\le|\nu|\le n$ and
$\ell(\nu)\le n-1$, the Schur form $s_\nu(\alpha,\omega)$ is
pointwise strongly positive.
\end{theorem}

\begin{proof}
We fix $p\in X$ and put $m:=|\nu|$.
For every $J\subseteq\{1,\ldots,n\}$ of size $m$, the angles
$(\phi_i)_{i\in J}$ are positive and have sum at most $\theta$.
Corollary~\ref{cor:schur-barycenter} therefore gives
\[
 C_\nu((\phi_i)_{i\in J})\ge0.
\]
By \eqref{eq:schur-squarefree-expansion}, the form $s_\nu(\alpha,\omega)$ is a
nonnegative linear combination of the strongly positive forms
$(\sqrt{-1})^{\,m^2}\mathrm{d}z^J\wedge\mathrm{d}\bar z^J$.
This completes the proof.
\end{proof}

\subsection{LYZ positive cones and comparisons}

We now pass from pointwise positivity to positivity of cohomological cones.
For $0\le k\le n$, let $\overline{\mathrm{Eff}}_{\mathrm{an}}^k(X)$ denote the closed
convex cone in $H^{k,k}(X,\R)$ generated by the cohomology classes
of effective analytic cycles of codimension $k$.
For any cone $\mathcal C$, the notation $S^m\mathcal C$ denotes the closed convex cone generated
by products of $m$ classes in $\mathcal C$.
The dual of $\overline{\mathrm{Eff}}_{\mathrm{an}}^{n-m}(X)$ is taken with respect to
\[
 H^{m,m}(X,\R)\times H^{n-m,n-m}(X,\R)\longrightarrow\R,
 \qquad (\xi,\eta)\longmapsto\int_X\xi\cup\eta.
\]

\begin{theorem}\label{thm:analytic-cone-inclusions}
Let $X$ be a compact K\"ahler manifold of dimension $n\ge2$. Then
\begin{equation}\label{eq:analytic-cone-inclusions}
 S^m\overline{\mathcal K(X)}\subseteq\croot^m(X)
 \subseteq\bigl(\overline{\mathrm{Eff}}_{\mathrm{an}}^{n-m}(X)\bigr)^\vee,
 \qquad 1\le m\le n.
\end{equation}
\end{theorem}

\begin{proof}
Let $H$ be a K\"ahler class. We choose a K\"ahler form $\alpha$
representing it. We set $\omega:=\tan(\pi/(2n))\alpha$.
All eigenangles of $\alpha$ relative to $\omega$ equal $\pi/(2n)$,
so $(\alpha,\omega)$ is a supercritical LYZ pair of phase $\pi/2$.
Proposition~\ref{prop:chern} gives
$c_1(\Rpair{[\alpha]}{[\omega]})=(n-1)H$.
Applying this construction to K\"ahler classes $H_1,\ldots,H_m$ yields
\[
 H_1\cdots H_m
 =\frac1{(n-1)^m}\prod_{i=1}^m c_1(\Rpair{[\alpha_i]}{[\omega_i]})
 \in\croot^m(X).
\]
Taking positive combinations and limits proves the first inclusion.

By Theorem~\ref{thm:root-schur-positive}, every mixed Schur product has a
smooth closed strongly positive representative, which pairs nonnegatively
with effective analytic cycles. This proves the second inclusion.
\end{proof}

\begin{corollary}
Every compact K\"ahler $n$-fold $X$, $n\ge2$, satisfies
$\croot^1(X)=\overline{\mathcal K(X)}$.
\end{corollary}
\begin{proof}
Every generator of $\croot^1(X)$ can be written as $c_1(\mathcal R)=(n-1)[\kappa]$, where
$[\kappa]$ is K\"ahler. Hence $\croot^1(X)\subseteq\overline{\mathcal K(X)}$.
The reverse inclusion is Theorem~\ref{thm:analytic-cone-inclusions}.
\end{proof}

\subsubsection{The pliant cone and the LYZ cone on threefolds}

We first express the preceding cone inclusions in terms of numerical
cycle classes. Let $X$ be a smooth projective variety of dimension $n$.
We write $N^k(X)_\R$ and $N_k(X)_\R$ for the real spaces of
codimension-$k$ and dimension-$k$ cycles modulo numerical equivalence.
The closed cone generated by effective $k$-dimensional cycles is denoted
by $\Psef_k(X)\subset N_k(X)_\R$. Its dual under the intersection
pairing is the nef cone
$\Nef^k(X):=(\Psef_k(X))^\vee\subset N^k(X)_\R$.
We also write $\Psef^k(X):=\Psef_{n-k}(X)$ under the identification
$N^k(X)_\R=N_{n-k}(X)_\R$.

To relate these numerical cones to cohomology, we set
\[
 H_{\mathrm{alg}}^{m,m}(X,\R)
 :=\operatorname{im}\!\left(
 \mathrm{cl}\colon\mathrm{CH}^m(X)_\R\longrightarrow H^{2m}(X,\R)
 \right),
\]
where $\mathrm{cl}$ is the cycle class map. The canonical surjection
$q\colon H_{\mathrm{alg}}^{m,m}(X,\R)\twoheadrightarrow N^m(X)_\R$
sends an algebraic cohomology class to its numerical class.

\begin{lemma}\label{lem:numerical-cone}
Let $X$ be a smooth projective variety of dimension $n\ge2$.
Then, for $1\le m\le n$,
\begin{equation}\label{rmk:numerical-cone}
 S^m\Nef^1(X)
 \subseteq q\!\left(\croot^m(X)\cap H_{\mathrm{alg}}^{m,m}(X,\R)\right)
 \subseteq\Nef^m(X).
\end{equation}
\end{lemma}
\begin{proof}
The right inclusion follows from \eqref{eq:analytic-cone-inclusions}
by intersection with effective cycles.
To prove the left inclusion, let $\gamma\in S^m\Nef^1(X)$.
We choose a sequence $\gamma_j\to\gamma$ of finite nonnegative sums
of products of ample divisor classes. Let $\xi_j$
be the corresponding cohomological classes. By
\eqref{eq:analytic-cone-inclusions},
$\xi_j\in\croot^m(X)\cap H_{\mathrm{alg}}^{m,m}(X,\R)$.
We fix an ample divisor class $H$. Pairing positive smooth representatives
of the classes $\xi_j$
with smooth representatives of a fixed dual cohomology basis gives
\[
 \|\xi_j\|\le C\int_X\xi_j\cup H^{n-m}
 =C\,\gamma_j\cdot H^{n-m}.
\]
Here we use a fixed cohomology norm $\|\cdot\|$, and the constant $C$ is independent of $j$.
The right-hand side is bounded because $\gamma_j\to\gamma$.
Thus a subsequence of $\xi_j$ converges to a class
$\xi\in\croot^m(X)\cap H_{\mathrm{alg}}^{m,m}(X,\R)$.
Continuity of $q$ gives $q(\xi)=\gamma$, proving the left inclusion.
\end{proof}

Now let $X$ be a smooth projective complex threefold.
In degree two, the map $q$ is an isomorphism.
Indeed, for an integral ample divisor class $H$, Hard Lefschetz and the Lefschetz
$(1,1)$ theorem give $H_{\mathrm{alg}}^{2,2}(X,\R)=H\cdot N^1(X)_\R$.
The Hodge index theorem makes the pairing $HAB$ nondegenerate on
$N^1(X)_\R$, so the canonical surjection is injective in this case.
We therefore identify $N^2(X)_\R$ with $H_{\mathrm{alg}}^{2,2}(X,\R)$.
Under this identification, we will strengthen the inclusions \eqref{rmk:numerical-cone}
by proving
$\PL^2(X)\subseteq\croot^2(X)\cap N^2(X)_\R$.

By \cite[Definition~3.2]{FL}, the pliant cone in degree two is
\[
 \PL^2(X):=\overline{\Cone\!\left\{
 s_{(1,1)}(E),\ s_{(2)}(E),\ c_1(E)c_1(F)
 \;:\; E,F\text{ are globally generated bundles}\right\}}.
\]
The following lemma reduces these generators to second Chern classes.

\begin{lemma}\label{lem:rank-reduction}
On a smooth projective threefold $X$,
\[
 \PL^2(X)=\overline{\Cone\{c_2(E):E\text{ is a globally generated vector bundle}\}}.
\]
\end{lemma}
\begin{proof}
Every class on the right lies in $\PL^2(X)$ by definition. For the reverse inclusion, we express the other two types
of generators as second Chern classes of globally generated bundles.
For a globally generated bundle $E$, let $M_E$ denote the kernel of its
evaluation map, so that
\[
 0\longrightarrow M_E\longrightarrow H^0(E)\otimes\OX
 \longrightarrow E\longrightarrow0.
\]
Dualizing gives a
surjection $H^0(E)^*\otimes\OX\twoheadrightarrow M_E^\vee$.
Thus the dual bundle $M_E^\vee$ is globally generated. The Whitney formula gives
\[
 c_2(M_E^\vee)=c_2(M_E)=c_1(E)^2-c_2(E).
\]
For globally generated bundles $E$ and $F$, the line bundles $\det E$ and $\det F$
are globally generated, so their direct sum is globally generated and
\[
 c_2(\det E\oplus\det F)=c_1(E)c_1(F).
\]
This completes the proof.
\end{proof}

The following lemma gives a criterion for a product of two divisor classes
to lie in $\croot^2(X)$.

\begin{lemma}\label{lem:phase}
Let $A,B\in N^1(X)_\R$. Suppose that $B+A$ and $B-A$ are ample, $AB\in\operatorname{int}\Nef^2(X)$, and
\[
A^2(3B-A)>0,\qquad B^2(B-3A)>0.
\]
Then $AB\in\croot^2(X)$.
\end{lemma}

\begin{proof}
The two positive intersection numbers in the hypotheses determine
a unique $\theta\in(0,\pi)$ by
\[
 \cos\theta=
 \frac{A^2(3B-A)-B^2(B-3A)}
      {A^2(3B-A)+B^2(B-3A)}.
\]
We set
\[
 [\kappa]:=\frac{A+B}{2},\qquad
 [\omega]:=\frac{B-A}{2}\sin\theta,\qquad
 [\alpha]:=[\kappa]+\frac{B-A}{2}\cos\theta.
\]
Since $[\omega]/\sin\theta=(B-A)/2$, we have
\[
\begin{aligned}
 4\bigl([\kappa]^3-3[\kappa]([\omega]/\sin\theta)^2\bigr)
 &=A^2(3B-A)-B^2(B-3A),\\
 8([\omega]/\sin\theta)^3
 &=(B-A)^3=A^2(3B-A)+B^2(B-3A).
\end{aligned}
\]
By the choice of $\theta$ and the identity
\eqref{eq:recover-costheta}, these equalities give
$\int_XP_3([\kappa],[\omega])=0$. By \eqref{eq:threefold-root-classes}, the associated LYZ root classes are
$A,B$, so their Chern classes are $c_1=A+B$ and $c_2=AB$.
The assumptions make $[\omega]$ and $c_1$ ample, while $c_2$
pairs positively with every irreducible surface.
Corollary~\ref{cor:root-chern-positive} therefore shows that $([\alpha],[\omega])$ constructed as above is a cohomological supercritical LYZ pair with $c_2=AB$. This completes the proof.
\end{proof}

To apply Lemma~\ref{lem:phase} to a curve class, we fix an ample class $H$.
The Hodge index theorem makes multiplication by $H$ an isomorphism
\[
 L_H\colon N^1(X)_\R\longrightarrow N_1(X)_\R,
 \qquad D\longmapsto DH.
\]
Following the notation in \cite[Definition~3.2]{PRT}, for a curve class
$\Gamma$, we write
\[
\frac{\Gamma}{H}:=L_H^{-1}(\Gamma),
\qquad\text{so that}\qquad
H\frac{\Gamma}{H}=\Gamma.
\]
\begin{definition}

Let $X$ be a smooth projective complex threefold. Let
$\Gamma\in N_1(X)_{\R}$ be a curve class and $H$ be an ample $\R$-divisor. The pair $(\Gamma,H)$ is called a \emph{Hodge--Riemann pair} if
\[
H\Gamma>0
\]
and
\[
H\left(\frac{\Gamma}{H}\right)^2>0.
\]
\end{definition}

\begin{corollary}\label{cor:HR-reconstruction}
Let $\Gamma\in\operatorname{int}\Nef^2(X)$ and let $H$ be ample. If
\[
 H\left(\frac{\Gamma}{H}\right)^2>0,
\]
then $\Gamma\in\croot^2(X)$.
\end{corollary}
\begin{proof}
We put $A:=u\Gamma/H$ and $B:=H$. For small $u>0$, the classes $H\pm A$
are ample and $AB=u\Gamma\in\operatorname{int}\Nef^2(X)$. Moreover,
\[
\begin{aligned}
 \frac{A^2(3H-A)}{u^2}
 &=\left(\frac{\Gamma}{H}\right)^2
   \left(3H-u\frac{\Gamma}{H}\right)
 \longrightarrow 3H\left(\frac{\Gamma}{H}\right)^2>0,\\
 H^2(H-3A)&=H^3-3uH\Gamma\longrightarrow H^3>0.
\end{aligned}
\]
Lemma~\ref{lem:phase} realizes $u\Gamma$ as the second Chern class
of a collection of LYZ root classes. Since $\croot^2(X)$ is a cone, it also contains $\Gamma$.
\end{proof}

We next apply the Hodge--Riemann property of ample vector bundles
to the criterion in Corollary~\ref{cor:HR-reconstruction}.
We use the following threefold Chern-class case of
\cite[Proposition~3.8]{PRT}.

\begin{proposition}[\cite{PRT}, Proposition~3.8]\label{prop:PRT-threefold}
Let $X$ be a smooth projective complex threefold,
let $E$ be a vector bundle of rank $r\ge2$, and let
$H\in\Amp^1(X)$ and $\varepsilon\ge0$.
If the $\R$-twisted bundle
$E_\varepsilon:=E\otimes\mathcal O_X(\varepsilon H)$ is ample,
then $c_2(E_\varepsilon)\in\operatorname{int}\Nef^2(X)$ and
\[
 \bigl(c_2(E_\varepsilon),(r-1)c_1(E_\varepsilon)\bigr)
\]
is a Hodge--Riemann pair.
\end{proposition}

\begin{corollary}\label{cor:ample-threefold-HR}
Under the hypotheses of Proposition~\ref{prop:PRT-threefold}, the pair
$(c_2(E_\varepsilon),c_1(E_\varepsilon))$ is a Hodge--Riemann pair, and $c_2(E_\varepsilon)\in\croot^2(X)$.
\end{corollary}
\begin{proof}
By Proposition~\ref{prop:PRT-threefold}, the pair
$(c_2(E_\varepsilon),(r-1)c_1(E_\varepsilon))$ is a Hodge--Riemann pair, and by rescaling the second component by $1/(r-1)>0$ it is clear that 
$(c_2(E_\varepsilon),c_1(E_\varepsilon))$ is again a Hodge--Riemann pair.
Moreover, by Proposition~\ref{prop:PRT-threefold} we have $c_2(E_\varepsilon)\in\operatorname{int}\Nef^2(X)$, and $c_1(E_\varepsilon)$ is ample since $E_\varepsilon$ is ample.
Then we can apply Corollary~\ref{cor:HR-reconstruction} to
$(c_2(E_\varepsilon),c_1(E_\varepsilon))$ to obtain
$c_2(E_\varepsilon)\in\croot^2(X)$.
\end{proof}

Approximating globally generated bundles by ample twists now yields
the desired cone inclusion.

\begin{theorem}\label{thm:pliant-inclusion}
For every smooth projective complex threefold $X$,
\begin{equation}
 S^2\Nef^1(X)\subseteq\PL^2(X)
 \subseteq\bigl(\croot^2(X)\cap N^2(X)_\R\bigr)
 \subseteq\Nef^2(X).
\end{equation}
\end{theorem}
\begin{proof}
The inclusion $S^2\Nef^1(X)\subseteq\PL^2(X)$ follows from
$c_1(L)c_1(M)=c_2(L\oplus M)$ for globally generated line bundles
$L,M$.

For $\PL^2(X)\subseteq\croot^2(X)$, Lemma~\ref{lem:rank-reduction}
reduces the problem to a globally generated bundle $E$ of rank $r\ge2$. We fix an integral ample
divisor $H$ and for $\varepsilon>0$ we set $E_\varepsilon:=E\otimes\mathcal O_X(\varepsilon H)$. Since a globally generated bundle is nef,
the twist $E_\varepsilon$ is an ample $\R$-twisted vector bundle.
Corollary~\ref{cor:ample-threefold-HR} therefore gives
$c_2(E_\varepsilon)\in\croot^2(X)$.

A direct computation shows that
\[
c_2(E_\varepsilon)=c_2(E)+(r-1)\varepsilon Hc_1(E)
                       +\binom r2\varepsilon^2H^2.
\]
We have $c_2(E_\varepsilon)\to c_2(E)$ as $\varepsilon\to0^+$, and thus $c_2(E)\in\croot^2(X)$. This shows the second inclusion. The last inclusion follows from Lemma~\ref{lem:numerical-cone}.
\end{proof}

\Needspace{6\baselineskip}
\subsubsection{Strict inclusion on the Fujino--Sato threefold}\label{subsec:strict-inclusion}
Let $X=X_{FS}$ be the smooth projective toric threefold of
\cite[Example~4.2]{FS}. Its cohomology is generated by divisor classes
\cite[Theorem~12.4.4]{CLS}, so $H^{2,2}(X,\R)=N^2(X)_\R$.

We prove that $\PL^2(X)\subseteq\croot^2(X)$ is strict by constructing
a special class $\Gamma_*$ in the LYZ cone. If $\Gamma_*$ belonged to
the pliant cone $\PL^2(X)$, Lemma~\ref{lem:supporting-face-rigidity}
would give a sequence of Chern classes with prescribed limiting properties.
Lemma~\ref{lem:small-chern} would then yield a special factorization of
$\Gamma_*$, contradicting its construction.
Appendix~\ref{app:FS-data} gives the fan and intersection data used below.

\paragraph{The candidate class}
Let $N_1,\ldots,N_5$ be the nef basis in \eqref{eq:FS-psef-basis},
and let $D_1,\ldots,D_8$ denote the
invariant prime divisors, numbered as in \cite[Example~4.2]{FS}.
The cone computations in Appendix~\ref{app:FS-data} give
\begin{equation}\label{eq:FS-cones}
 \begin{aligned}
 \Nef^1(X)&=\Cone(N_1,\ldots,N_5),\\
 \Psef^1(X)&=\Cone(D_1,D_5,D_6,D_7,D_8).
 \end{aligned}
\end{equation}
We set $K:=N_1+\cdots+N_5$, an integral ample class.
For our example, we choose the integral divisor $M$ and define
$\Gamma_*$ by
\begin{align}
 M&:=-3N_1+2N_2-4N_3+N_4+6N_5,\\
 \Gamma_*&:=N_4M.
\end{align}
The values in \eqref{eq:FS-M-squares} and
\eqref{eq:FS-triple-values} give
\begin{equation}\label{eq:witness-numerics}
 M^2N_4=0,\qquad M^2N_i>0\quad(i\ne4),\qquad MK^2>0.
\end{equation}
Thus the curve class $M^2$ vanishes precisely on the nef ray generated by $N_4$.
In particular, $M\Gamma_*=0$ and $M^2H>0$ for every ample class $H$.

Equation~\eqref{eq:FS-exact-certificate} shows that the generators
of $\Psef^1(X)$ pair nonnegatively with $\Gamma_*$ and strictly
positively with $KM$. Moreover, $D_1\Gamma_*=2$, so $\Gamma_*\ne0$.
By \eqref{eq:FS-cones} and the duality
$\Nef^2(X)=(\Psef^1(X))^\vee$, we obtain
\begin{equation}\label{eq:witness-curve-positivity}
 0\ne\Gamma_*\in\Nef^2(X),\qquad
 KM\in\operatorname{int}\Nef^2(X).
\end{equation}
Since the ample cone is open, $K-\varepsilon D_1$ is ample, hence
pseudoeffective, for sufficiently small $\varepsilon>0$.
Thus $(K-\varepsilon D_1)\Gamma_*\ge0$, and
\[
 K\Gamma_*\ge\varepsilon D_1\Gamma_*=2\varepsilon>0,
\]
where $D_1\Gamma_*=2$ follows from \eqref{eq:FS-exact-certificate}
in Appendix~\ref{app:FS-data}.

\Needspace{4\baselineskip}
\begin{lemma}\label{lem:witness-LYZ}
The class $\Gamma_*=N_4M$ belongs to $\croot^2(X)$.
\end{lemma}
\begin{proof}
For $u>0$, we set
\[
 H_u:=N_4+uK,
 \qquad
 \Gamma_u:=H_uM=\Gamma_*+uKM.
\]
The class $H_u$ is ample, and \eqref{eq:witness-curve-positivity}
gives $\Gamma_u\in\operatorname{int}\Nef^2(X)$.
Moreover, $\Gamma_u/H_u=M$ and
$H_uM^2=uKM^2>0$ by \eqref{eq:witness-numerics}.
Corollary~\ref{cor:HR-reconstruction} gives
$\Gamma_u\in\croot^2(X)$, and closedness gives
$\Gamma_*\in\croot^2(X)$ as $u\to0^+$.
\end{proof}

\Needspace{4\baselineskip}
\paragraph{Limits of normalized Chern classes}
The weak Hodge--Riemann inequality for globally generated bundles
shows that $M$ is nonnegative on the pliant cone.

\begin{lemma}\label{lem:M-supports-pliant}
The divisor $M$ satisfies $M\gamma\ge0$ for every
$\gamma\in\PL^2(X)$, with $M\Gamma_*=0$.
\end{lemma}
\begin{proof}
We have $M\Gamma_*=M^2N_4=0$ by \eqref{eq:witness-numerics}.
By Lemma~\ref{lem:rank-reduction}, it remains to prove $Mc_2(E)\ge0$
for every globally generated bundle $E$ with $c_2(E)\ne0$.
For such a bundle $E$ and a divisor class $A\in N^1(X)_\R$,
we first show that
\begin{equation}\label{eq:weak-Chern-HR}
 Ac_2(E)=0\quad\Longrightarrow\quad A^2c_1(E)\le0.
\end{equation}
Since $E$ is globally generated, the class $c_2(E)$ is represented by an
effective curve cycle \cite[Example~14.4.3(i)]{Fulton}.
As $c_2(E)\ne0$, ampleness of $K$ gives $Kc_2(E)>0$.

For $\varepsilon>0$, the twist
$E_\varepsilon:=E\otimes\mathcal O_X(\varepsilon K)$ is an ample $\R$-twisted bundle,
and Proposition~\ref{prop:PRT-threefold} gives
$c_2(E_\varepsilon)\in\operatorname{int}\Nef^2(X)$.
Since $K$ is ample and hence pseudoeffective, we obtain
$Kc_2(E_\varepsilon)>0$.
If $Ac_2(E)=0$, we define
\[
 A_\varepsilon:=A-
 \frac{Ac_2(E_\varepsilon)}{Kc_2(E_\varepsilon)}K.
\]
A direct computation gives
\[
 A_\varepsilon c_2(E_\varepsilon)
 =Ac_2(E_\varepsilon)-
 \frac{Ac_2(E_\varepsilon)}{Kc_2(E_\varepsilon)}
 Kc_2(E_\varepsilon)=0.
\]
Then Corollary~\ref{cor:ample-threefold-HR} gives
$A_\varepsilon^2c_1(E_\varepsilon)\le0$.
Moreover, $Ac_2(E_\varepsilon)\to Ac_2(E)=0$ and
$Kc_2(E_\varepsilon)\to Kc_2(E)>0$, so $A_\varepsilon\to A$.
We obtain \eqref{eq:weak-Chern-HR} by passing to the limit.

Since $E$ is globally generated, the class $c_1(E)$ is nef.
Hence $c_1(E)M^2\ge0$ by \eqref{eq:witness-numerics}.
By the definition of the pliant cone and
\cite[Lemma~3.7]{FL}, we have
\[
 c_1(E)^2-c_2(E)=s_{(2)}(E)\in\PL^2(X)\subseteq\Nef^2(X).
\]
Pairing with $K$ gives $Kc_1(E)^2\ge Kc_2(E)>0$, so $c_1(E)\ne0$.
The nonzero nef class $c_1(E)$ is pseudoeffective.
Since $KM\in\operatorname{int}\Nef^2(X)$ by
\eqref{eq:witness-curve-positivity}, we obtain $c_1(E)MK>0$.

If $Mc_2(E)<0$, we set $t:=-Mc_2(E)/(Kc_2(E))>0$.
Then $(M+tK)c_2(E)=0$, but
\[
 (M+tK)^2c_1(E)
 =c_1(E)M^2+2t c_1(E)MK+t^2c_1(E)K^2>0,
\]
contrary to \eqref{eq:weak-Chern-HR}.
Hence $Mc_2(E)\ge0$. This completes the proof.
\end{proof}

We will use this supporting inequality to constrain the Chern limits of
bundles whose second Chern classes approach $M\gamma=0$.

\Needspace{4\baselineskip}
\begin{lemma}\label{lem:supporting-face-rigidity}
Assume that $\Gamma_*\in\PL^2(X)$.
Then there are globally generated bundles $E_j$ of ranks at least two,
with $c_1(E_j)=\sum_{i=1}^5h_{i,j}N_i$, $h_{i,j}\ge0$,
and $\sum_{i=1}^5h_{i,j}>0$, such that
\begin{equation}
 \frac{c_1(E_j)}{\sum_{\ell=1}^5h_{\ell,j}}\longrightarrow N_4,
 \qquad
 \frac{c_2(E_j)}{Kc_2(E_j)}\longrightarrow\frac{\Gamma_*}{K\Gamma_*},
 \qquad
 \frac{Kc_2(E_j)}{\bigl(\sum_{\ell=1}^5h_{\ell,j}\bigr)^2}
 \longrightarrow0.
\end{equation}
\end{lemma}
\begin{proof}
By Lemma~\ref{lem:rank-reduction}, the cone $\PL^2(X)$ is the closed convex
cone generated by the classes $c_2(E)$ of globally generated bundles.
Since $\Gamma_*/(K\Gamma_*)\in\PL^2(X)$, there are finite sums
\[
 G_j:=\sum_i a_{j,i}c_2(E_{j,i})\longrightarrow\frac{\Gamma_*}{K\Gamma_*},
 \qquad a_{j,i}>0,
\]
where each $E_{j,i}$ is globally generated and $c_2(E_{j,i})\ne0$.
In particular, $\operatorname{rk}(E_{j,i})\ge2$.
Each $c_2(E_{j,i})$ is represented by a nonzero effective curve cycle,
so ampleness of $K$ gives $Kc_2(E_{j,i})>0$.
Lemma~\ref{lem:M-supports-pliant} also gives $Mc_2(E_{j,i})\ge0$.
Moreover, $KG_j\to1$ and $MG_j\to0$.

We choose $E_j$ so that
$\frac{Mc_2(E_j)}{Kc_2(E_j)}
=\min_i\bigl\{\frac{Mc_2(E_{j,i})}{Kc_2(E_{j,i})}\bigr\}$.
Then
\[
 0\le\frac{Mc_2(E_j)}{Kc_2(E_j)}
 \le\sum_i\frac{a_{j,i}Kc_2(E_{j,i})}{KG_j}
             \frac{Mc_2(E_{j,i})}{Kc_2(E_{j,i})}
 =\frac{MG_j}{KG_j}\longrightarrow0.
\]
By compactness of $\{\eta\in\Nef^2(X):K\eta=1\}$,
we pass to a subsequence such that
\[
 \frac{c_2(E_j)}{Kc_2(E_j)}\longrightarrow\Gamma,
 \qquad K\Gamma=1,\qquad M\Gamma=0.
\]

Since $E_j$ is globally generated, the class $c_1(E_j)$ is nef.
We write $c_1(E_j)=\sum_{i=1}^5h_{i,j}N_i$, with $h_{i,j}\ge0$.
The class $c_1(E_j)^2-c_2(E_j)$ is pliant and hence nef by
\cite[Lemma~3.7]{FL}. Pairing with $K$ gives
\[
 0<Kc_2(E_j)\le Kc_1(E_j)^2.
\]
Thus $c_1(E_j)\ne0$ and $\sum_{\ell=1}^5h_{\ell,j}>0$.
The vectors
$\bigl(\frac{h_{1,j}}{\sum_{\ell=1}^5h_{\ell,j}},\ldots,
\frac{h_{5,j}}{\sum_{\ell=1}^5h_{\ell,j}}\bigr)$
have nonnegative entries summing to one, so they lie in the compact
standard simplex
\[
 \left\{(x_1,\ldots,x_5)\in\R_{\ge0}^5:
 \sum_{i=1}^5x_i=1\right\}.
\]
Within the subsequence along which
$c_2(E_j)/(Kc_2(E_j))\to\Gamma$,
we therefore choose a further subsequence indexed by $j_k$ such that
\[
 \frac{h_{i,j_k}}{\sum_{\ell=1}^5h_{\ell,j_k}}
 \longrightarrow h_i\quad(i=1,\ldots,5),
 \qquad h_i\ge0,\qquad \sum_{i=1}^5h_i=1.
\]
Relabelling $j_k$ as $j$, we obtain
\[
 \frac{c_1(E_j)}{\sum_{\ell=1}^5h_{\ell,j}}\longrightarrow
 H:=\sum_{i=1}^5h_iN_i.
\]

We claim that relation \eqref{eq:weak-Chern-HR} holds for the limit pair
$(\Gamma,H)$.
Indeed, suppose that a divisor class $A$ satisfies $A\Gamma=0$.
We define $A_j:=A-\frac{Ac_2(E_j)}{Kc_2(E_j)}K$. Then
\[
 A_j\longrightarrow A,\qquad A_jc_2(E_j)=0,\qquad
 A_j^2\frac{c_1(E_j)}{\sum_{\ell=1}^5h_{\ell,j}}\le0
\]
by \eqref{eq:weak-Chern-HR}. Passing to the limit gives
\begin{equation}\label{eq:limiting-Chern-HR}
 A\Gamma=0\quad\Longrightarrow\quad A^2H\le0.
\end{equation}

Inequality \eqref{eq:limiting-Chern-HR} first determines the limit
$H$ of the normalized first Chern classes.
Since $M\Gamma=0$, we can take $A=M$ in
\eqref{eq:limiting-Chern-HR}. Using $M^2N_4=0$ and
$M^2N_i>0$ for $i\ne4$ from \eqref{eq:witness-numerics}, we obtain
\[
 0\ge M^2H=\sum_{i\ne4}h_i(M^2N_i)\ge0.
\]
Every summand must vanish. Thus $h_i=0$ for $i\ne4$, and
$\sum_i h_i=1$ gives $h_4=1$. Therefore
\[
 H=N_4.
\]

Inequality \eqref{eq:limiting-Chern-HR}, together with $H=N_4$,
also determines the second Chern limit $\Gamma$.
We claim that every divisor class $A$
with $A\Gamma=0$ also satisfies $A\Gamma_*=0$.
Indeed, for such an $A$, we have
$(M+tA)\Gamma=0$ for every $t\in\R$. Applying
\eqref{eq:limiting-Chern-HR} with $H=N_4$ and using
$N_4M^2=0$ and $N_4M=\Gamma_*$ gives
\[
 \begin{aligned}
 0&\ge N_4(M+tA)^2\\
  &=N_4M^2+2tN_4MA+t^2N_4A^2\\
  &=2t\Gamma_*A+t^2N_4A^2.
 \end{aligned}
\]
Dividing by $t>0$ and letting $t\to0^+$ gives $\Gamma_*A\le0$.
Dividing by $t<0$ reverses the inequality, and letting $t\to0^-$
gives $\Gamma_*A\ge0$. Thus $\Gamma_*A=0$.
By the construction of $\Gamma$, we have $K\Gamma=1$.
For any divisor class $A$, the class $A-(A\Gamma)K$ therefore has zero
intersection with $\Gamma$. By the preceding claim, we have
\[
 0=(A-(A\Gamma)K)\Gamma_*
   =A\Gamma_*-(A\Gamma)K\Gamma_*.
\]
It follows that $A\Gamma=A\Gamma_*/(K\Gamma_*)$ for every divisor class $A$.
By numerical equivalence of curve classes, this means
$\Gamma=\Gamma_*/(K\Gamma_*)$.

It remains to prove
$Kc_2(E_j)/(\sum_{\ell=1}^5h_{\ell,j})^2\to0$.
We use the effective divisor $D_8$.
The values in \eqref{eq:FS-exact-certificate} and
\eqref{eq:FS-triple-values} give
\begin{equation}\label{eq:witness-D8}
 D_8N_4^2=0,\qquad D_8\Gamma_*>0.
\end{equation}
Both $c_2(E_j)$ and $c_1(E_j)^2-c_2(E_j)$ are pliant and hence nef.
Since $D_8$ is effective, pairing these two classes with $D_8$ gives
$0\le D_8c_2(E_j)\le D_8c_1(E_j)^2$, and thus
\[
 0\le\frac{Kc_2(E_j)}{\bigl(\sum_{\ell=1}^5h_{\ell,j}\bigr)^2}
       \frac{D_8c_2(E_j)}{Kc_2(E_j)}
 \le D_8\left(\frac{c_1(E_j)}{\sum_{\ell=1}^5h_{\ell,j}}\right)^2.
\]
Since $c_2(E_j)/(Kc_2(E_j))\to\Gamma_*/(K\Gamma_*)$
and $c_1(E_j)/(\sum_{\ell=1}^5h_{\ell,j})\to N_4$,
by \eqref{eq:witness-D8} we obtain
\[
 \frac{D_8c_2(E_j)}{Kc_2(E_j)}
 \longrightarrow\frac{D_8\Gamma_*}{K\Gamma_*}>0,\qquad
 D_8\left(\frac{c_1(E_j)}{\sum_{\ell=1}^5h_{\ell,j}}\right)^2
 \longrightarrow D_8N_4^2=0.
\]
In particular, $D_8c_2(E_j)>0$ for all sufficiently large $j$, so
\[
 0\le\frac{Kc_2(E_j)}{\bigl(\sum_{\ell=1}^5h_{\ell,j}\bigr)^2}
 \le D_8\left(\frac{c_1(E_j)}{\sum_{\ell=1}^5h_{\ell,j}}\right)^2
       \frac{Kc_2(E_j)}{D_8c_2(E_j)}
 \longrightarrow0.\qedhere
\]
\end{proof}

\Needspace{4\baselineskip}
\paragraph{A decomposition of the limit classes}

The following lemma turns the growth condition in
Lemma~\ref{lem:supporting-face-rigidity} into a factorization of $\Gamma_*$ by a
pseudoeffective divisor.

\Needspace{4\baselineskip}
\begin{lemma}\label{lem:small-chern}
Let $X$ be a smooth projective complex threefold, let $K$ be an
integral ample divisor, and let $E_j$ be globally generated vector
bundles of ranks $r_j\ge2$ with $Kc_2(E_j)>0$.
Assume that $H\in N^1(X)_\R$, $\Gamma\in N^2(X)_\R$,
and $m_j>0$ satisfy $H^2K>0$ and
\[
 \frac{c_1(E_j)}{m_j}\longrightarrow H,\qquad
 \frac{c_2(E_j)}{Kc_2(E_j)}\longrightarrow\Gamma,\qquad
 \frac{Kc_2(E_j)}{m_j^2}\longrightarrow0.
\]
Then $\Gamma=H B$ for some $B\in\Psef^1(X)$.
\end{lemma}
\begin{proof}
We claim that it suffices to construct effective divisors $B_j$ such
that $m_jB_j/(Kc_2(E_j))$ has a convergent subsequence and
$(c_2(E_j)-c_1(E_j)B_j)/(Kc_2(E_j))\to0$.
Indeed, we pass to a subsequence with $m_jB_j/(Kc_2(E_j))\to B$.
Since each $B_j$ is effective, the limit $B$ lies in $\Psef^1(X)$.
Using $c_1(E_j)/m_j\to H$, we obtain
\[
 \frac{c_2(E_j)}{Kc_2(E_j)}
 =\underbrace{\frac{c_1(E_j)}{m_j}\frac{m_jB_j}{Kc_2(E_j)}}
     _{\longrightarrow H B}
  +\underbrace{\frac{c_2(E_j)-c_1(E_j)B_j}{Kc_2(E_j)}}
     _{\longrightarrow 0}
 \longrightarrow H B.
\]
By hypothesis, we also have $c_2(E_j)/(Kc_2(E_j))\to\Gamma$,
and this gives $\Gamma=H B$.

To construct the desired divisors $B_j$, we first reduce to the case where
the sheaves $E_j$ are reflexive of rank two.
The global sections of $E_j$ span every fibre, so $h^0(X,E_j)\ge r_j$
and a general choice of $r_j-2$ global sections is linearly
independent at a general point of $X$.
These sections give an exact sequence
\[
 0\longrightarrow\mathcal O_X^{\,r_j-2}
 \longrightarrow E_j\longrightarrow F_j\longrightarrow0,
\]
where $F_j$ is globally generated of rank two, with
$c_1(F_j)=c_1(E_j)$ and $c_2(F_j)=c_2(E_j)$.
For $r_j=2$, we simply take $F_j=E_j$.
For $r_j\ge3$, general sections become linearly dependent only in
codimension at least three, so the quotient $F_j$ is locally free
away from finitely many points.
At each such point $x$, the exact sequence and the depth lemma give
$\operatorname{depth}_{\mathcal O_{X,x}}((F_j)_x)\ge2$.
Hence the sheaf $F_j$ is reflexive
\cite[Lemmas~10.72.6 and~31.13.10]{Stacks}.
Replacing $E_j$ by $F_j$, we may therefore assume that each $E_j$
is a globally generated reflexive sheaf of rank two, since its first
two Chern classes are unchanged.

We write $\Delta(E_j):=4c_2(E_j)-c_1(E_j)^2$ for the discriminant.
By hypothesis,
\[
 \frac{K\Delta(E_j)}{m_j^2}
 =4\frac{Kc_2(E_j)}{m_j^2}-\left(\frac{c_1(E_j)}{m_j}\right)^2K
 \longrightarrow-H^2K<0.
\]
After discarding finitely many terms, we have
\begin{equation}\label{eq:negative-discriminant}
 K\Delta(E_j)<0.
\end{equation}
The classical Bogomolov inequality therefore implies that every
$E_j$ is $\mu_{K^2}$-unstable.

Let $L_j\subset E_j$ be its saturated rank-one destabilizing subsheaf.
Since $E_j$ is reflexive, the saturated subsheaf $L_j$ is reflexive
\cite[Lemma~31.13.7]{Stacks}, hence a line bundle.
We choose a divisor $B_j$ such that
$\mathcal O_X(B_j)=\det(E_j)\otimes L_j^{-1}$.
The quotient $E_j/L_j$ is torsion-free of rank one with determinant
$\mathcal O_X(B_j)$, so it has the form $I_{Z_j}(B_j)$.
We obtain
\[
 0\longrightarrow\det(E_j)\otimes\mathcal O_X(-B_j)
 \longrightarrow E_j\longrightarrow I_{Z_j}(B_j)\longrightarrow0,
\]
where $Z_j$ has codimension at least two and $I_{Z_j}$ is its ideal sheaf.
Since $L_j$ destabilizes $E_j$, we have
\begin{equation}\label{eq:destabilizing-line-slope}
 (c_1(E_j)-2B_j)K^2
 =2\bigl(\mu_{K^2}(L_j)-\mu_{K^2}(E_j)\bigr)>0.
\end{equation}

\noindent Since $I_{Z_j}(B_j)$ is globally generated, its nonzero
global sections define a nonempty linear subsystem of $|B_j|$
with no base points outside $Z_j$.
Thus we may choose $B_j$ effective, and
$\operatorname{codim}Z_j\ge2$ implies that this subsystem has no fixed divisor.
Two general members therefore meet properly.
The class of their intersection cycle is $B_j^2$.
Their intersection contains $Z_j$ as a closed subscheme, so
$B_j^2-[Z_j]$ is effective, where $[Z_j]$ denotes the
codimension-two cycle of $Z_j$
\cite[Lemmas~42.29.5 and~42.10.4]{Stacks}.
In particular,
\begin{equation}\label{eq:Zbound}
 0\le K[Z_j]\le KB_j^2.
\end{equation}
By the exact sequence above, we have
\begin{equation}\label{eq:c2-decomp}
 c_2(E_j)=(c_1(E_j)-B_j)B_j+[Z_j]
        =c_1(E_j)B_j-B_j^2+[Z_j].
\end{equation}

\noindent Thus
\[
\begin{aligned}
 (c_1(E_j)-2B_j)^2K
 &=c_1(E_j)^2K-4c_1(E_j)B_jK+4B_j^2K\\
 &=c_1(E_j)^2K-4Kc_2(E_j)+4K[Z_j]\\
 &=-K\Delta(E_j)+4K[Z_j]>0,
\end{aligned}
\]
where the second equality follows from \eqref{eq:c2-decomp},
and the last inequality follows from \eqref{eq:negative-discriminant}
and $K[Z_j]\ge0$ in \eqref{eq:Zbound}.

We write
$c_1(E_j)-2B_j=a_jK+S_j$ and $B_j=b_jK+T_j$, where
$S_jK^2=T_jK^2=0$.
The Hodge index theorem and Cauchy--Schwarz give
$S_jT_jK\ge-\sqrt{(-S_j^2K)(-T_j^2K)}$, and thus
\[
\begin{aligned}
 (c_1(E_j)-2B_j)B_jK
 &=a_jb_jK^3+S_jT_jK\\
 &\ge a_jb_jK^3-\sqrt{(-S_j^2K)(-T_j^2K)}.
\end{aligned}
\]
By \eqref{eq:destabilizing-line-slope}, we have
$a_jK^3=(c_1(E_j)-2B_j)K^2>0$, so $a_j>0$.
The effectivity of $B_j$ gives $b_jK^3=B_jK^2\ge0$, so $b_j\ge0$.
The positivity of $(c_1(E_j)-2B_j)^2K=a_j^2K^3+S_j^2K$ gives
$-S_j^2K<a_j^2K^3$.
Similarly, \eqref{eq:Zbound} gives $B_j^2K=b_j^2K^3+T_j^2K\ge0$,
hence $-T_j^2K\le b_j^2K^3$.
Combining these inequalities gives
\[
 (c_1(E_j)-2B_j)B_jK
 \ge a_jb_jK^3-\sqrt{(-S_j^2K)(-T_j^2K)}\ge0.
\]
Thus $B_j^2K\le\tfrac12c_1(E_j)B_jK$, and
\begin{equation}\label{eq:degree-lower}
 Kc_2(E_j)=c_1(E_j)B_jK-B_j^2K+K[Z_j]
 \ge\tfrac12c_1(E_j)B_jK.
\end{equation}

\noindent Each class $c_1(E_j)/m_j$ is nef. Hence the limit $H$ is nef and
satisfies $H K^2\ge0$. Since $H^2K>0$ by hypothesis, the Hodge index theorem gives
$HK^2\ge\sqrt{(H^2K)K^3}>0$.
We can therefore fix $\varepsilon>0$ sufficiently small that, for all
sufficiently large $j$,
\[
 \left(\frac{c_1(E_j)}{m_j}-\varepsilon K\right)^2K>0,
 \qquad
 \left(\frac{c_1(E_j)}{m_j}-\varepsilon K\right)K^2>0.
\]
We write
\[
 \frac{c_1(E_j)}{m_j}-\varepsilon K=a'_jK+S'_j,
 \qquad S'_jK^2=0,
\]
and use the decomposition $B_j=b_jK+T_j$ above.
The preceding inequalities give $a'_j>0$ and
$-(S'_j)^2K<(a'_j)^2K^3$.
Moreover, $b_j\ge0$ and $-T_j^2K\le b_j^2K^3$, as shown above.
By the Hodge index theorem and Cauchy--Schwarz, we have
\[
\begin{aligned}
 \left(\frac{c_1(E_j)}{m_j}-\varepsilon K\right)B_jK
 &=a'_jb_jK^3+S'_jT_jK\\
 &\ge a'_jb_jK^3-\sqrt{\bigl(-(S'_j)^2K\bigr)(-T_j^2K)}\\
 &\ge0.
\end{aligned}
\]
Together with \eqref{eq:degree-lower}, this yields
\[
 0\le K^2\frac{m_jB_j}{Kc_2(E_j)}
 \le\frac{c_1(E_j)B_jK}{\varepsilon Kc_2(E_j)}
 \le\frac2\varepsilon.
\]
On $\Psef^1(X)$, the $K^2$-degree is the restriction of a norm on
$N^1(X)_\R$ \cite[Theorem~1.4(3)]{FL}.
This inequality therefore makes $m_jB_j/(Kc_2(E_j))$ a bounded sequence
in a finite-dimensional space. Passing to a subsequence, we obtain
$m_jB_j/(Kc_2(E_j))\to B$.

For the remaining estimate, \eqref{eq:c2-decomp} gives
$c_2(E_j)-c_1(E_j)B_j=-(B_j^2-[Z_j])$.
It therefore suffices to prove that
$(B_j^2-[Z_j])/(Kc_2(E_j))\to0$.
On $\Psef^2(X)$, the $K$-degree is the restriction of a norm on
$N^2(X)_\R$ \cite[Theorem~1.4(3)]{FL}.
Since these classes are effective, to prove $(B_j^2-[Z_j])/(Kc_2(E_j))\to0$,
it suffices to prove that their $K$-degrees tend to zero.

The Hodge index theorem gives
$(c_1(E_j)B_jK)^2\ge(c_1(E_j)^2K)(B_j^2K)$.
Combining this with $c_1(E_j)B_jK\le2Kc_2(E_j)$ from \eqref{eq:degree-lower},
we obtain
\[
 0\le\frac{K(B_j^2-[Z_j])}{Kc_2(E_j)}
 \le\frac{KB_j^2}{Kc_2(E_j)}
 \le\frac{4Kc_2(E_j)}{c_1(E_j)^2K}
 =\frac{4Kc_2(E_j)/m_j^2}{(c_1(E_j)/m_j)^2K}
 \longrightarrow0.
\]
By the reduction at the beginning of the proof, the lemma follows.
\end{proof}

We now apply this factorization to the Fujino--Sato threefold. The matrix
\eqref{eq:N4-intersection-matrix}, the basis change
\eqref{eq:FS-psef-basis}, and the values in
\eqref{eq:FS-triple-values} give, for the map
$N_4\cup-\colon N^1(X)_\R\to N^2(X)_\R$,
\begin{equation}\label{eq:N4-basic-data}
 \ker(N_4\cup-)=\R D_6,\qquad
 N_4^2K>0.
\end{equation}
The positivity allows us to apply Lemma~\ref{lem:small-chern}
with $H=N_4$, while the kernel determines all possible divisors $B$
in the resulting factorization.

\Needspace{4\baselineskip}
\begin{theorem}
On the Fujino--Sato threefold $X_{FS}$, the class
$\Gamma_*=N_4M$ belongs to
$\croot^2(X_{FS})\setminus\PL^2(X_{FS})$. In particular,
\[
 \PL^2(X_{FS})\subsetneq\croot^2(X_{FS}).
\]
\end{theorem}
\begin{proof}
By Lemma~\ref{lem:witness-LYZ}, $\Gamma_*\in\croot^2(X)$.
Suppose that $\Gamma_*\in\PL^2(X)$.
By Lemmas~\ref{lem:supporting-face-rigidity} and~\ref{lem:small-chern}, we have
\begin{equation}\label{eq:forbidden-factorization}
 \frac{\Gamma_*}{K\Gamma_*}=N_4B
 \qquad\text{for some }B\in\Psef^1(X).
\end{equation}
Since $\Gamma_*=N_4M$, equation \eqref{eq:forbidden-factorization}
implies $B-M/(K\Gamma_*)\in\ker(N_4\cup-)$.
By \eqref{eq:N4-basic-data}, this kernel is $\R D_6$, and thus
$B=M/(K\Gamma_*)+\lambda D_6$ for some $\lambda\in\R$.
Using \eqref{eq:FS-psef-basis}, we express $M$ in the
pseudoeffective basis as
\begin{equation}\label{eq:M-psef-coordinates}
 M=5D_1+7D_5+5D_6+2D_7-2D_8.
\end{equation}
\noindent Dividing \eqref{eq:M-psef-coordinates} by $K\Gamma_*$ and using
$B=M/(K\Gamma_*)+\lambda D_6$ gives
\[
 B=\frac5{K\Gamma_*}D_1+\frac7{K\Gamma_*}D_5
 +\left(\frac5{K\Gamma_*}+\lambda\right)D_6
 +\frac2{K\Gamma_*}D_7-\frac2{K\Gamma_*}D_8.
\]
\noindent By \eqref{eq:FS-cones}, the cone $\Psef^1(X)$ is generated by
the basis $D_1$, $D_5$, $D_6$, $D_7$, $D_8$.
Since $K\Gamma_*>0$, the negative $D_8$-coefficient contradicts
$B\in\Psef^1(X)$ in \eqref{eq:forbidden-factorization}.
Hence $\Gamma_*\notin\PL^2(X)$, and
Theorem~\ref{thm:pliant-inclusion} gives the strict inclusion.
\end{proof}

\Needspace{6\baselineskip}
\section{K\"ahler package from supercritical LYZ pairs}\label{sp:sec:small-phase}
In this section, we study the K\"ahler package for the Chern classes of
supercritical LYZ root classes. We first establish pointwise
Hodge--Riemann relations for the associated Chern forms, and then use
the techniques in \cite[Propositions~3.1--3.2]{DN} to prove the corresponding
cohomological results.

Let $([\alpha],[\omega])$ be a cohomological supercritical LYZ pair on a compact
K\"ahler $n$-fold $X$, and let $\mathcal R:=\Rpair{[\alpha]}{[\omega]}$.
We say that the Chern classes $c_m(\mathcal R)$ satisfy the K\"ahler package
if the following three properties hold for every bidegree $(p,q)$,
where $p,q\ge0$ and $1\le p+q\le n$.
\begin{enumerate}[label=(\arabic*),leftmargin=*]
\item \emph{Hard Lefschetz.}
Multiplication by $c_{n-p-q}(\mathcal R)$ gives an isomorphism
\[
 H^{p,q}(X)\longrightarrow H^{n-q,n-p}(X).
\]

\item \emph{Hodge--Riemann relations.}
We define the primitive space by
\[
 \mathcal P^{p,q}(X,\mathcal R):=\ker\bigl(c_{n-p-q+1}(\mathcal R)\cup-\colon
 H^{p,q}(X)\longrightarrow H^{n-q+1,n-p+1}(X)\bigr),
\]
where Chern classes of degree greater than $n-1$ are zero.
The same definition gives $\mathcal P^{0,0}(X,\mathcal R)=H^{0,0}(X)$.
For classes $a,b\in H^{p,q}(X)$, we put
\[
 Q_{\mathcal R}^{p,q}(a,b):=(\sqrt{-1})^{\,p-q}(-1)^{(p+q)(p+q-1)/2}
 \int_X a\cup\bar b\cup c_{n-p-q}(\mathcal R).
\]
The Hodge--Riemann relations require this form to be positive definite
on $\mathcal P^{p,q}(X,\mathcal R)$.

\item \emph{Orthogonal Lefschetz decomposition.}
For $p,q\ge1$, we set $m:=n-p-q+1$.
Assuming Hard Lefschetz in bidegree $(p,q)$, we define
\begin{equation}
 \mathsf L_{\mathcal R}\colon H^{p-1,q-1}(X)\longrightarrow H^{p,q}(X),
 \qquad
 \mathsf L_{\mathcal R}:=\bigl(c_{m-1}(\mathcal R)\cup-\bigr)^{-1}
 \circ\bigl(c_m(\mathcal R)\cup-\bigr).
\end{equation}
We require the decomposition
\[
 H^{p,q}(X)=\bigoplus_{i=0}^{\min\{p,q\}}
 \mathsf L_{\mathcal R}^{\,i}\mathcal P^{p-i,q-i}(X,\mathcal R).
\]
The primitive summand is $Q_{\mathcal R}^{p,q}$-orthogonal to all other summands.
When $p=0$ or $q=0$, the whole space is primitive.
\end{enumerate}
\begin{theorem*}[Theorem~\ref{sp:thm:global}, Corollary~\ref{cor:low-dimensional-package}]
For a cohomological supercritical LYZ pair on a compact K\"ahler $n$-fold,
the three properties above hold whenever $1\le p+q\le n$ and
$\min\{p,q\}\le1$.
As a consequence, the K\"ahler package holds when $n=3$ or $4$.
\end{theorem*}

\subsection{Pointwise Hodge--Riemann relations}
In this subsection, we prove the pointwise Hodge--Riemann relations.
Let $(\alpha,\omega)$ be LYZ representatives with phase $\theta\in(0,\pi)$,
and set
$\kappa:=\alpha-\cot\theta\,\omega=c_1(\alpha,\omega)/(n-1)$.
At a fixed point, an $\omega$-orthonormal coframe diagonalizing $\alpha$ gives
\[
 \alpha=\sqrt{-1}\sum_{i=1}^n\lambda_i\,\mathrm{d}z^i\wedge\mathrm{d}\bar z^i,
 \qquad
 \omega=\sqrt{-1}\sum_{i=1}^n\mathrm{d}z^i\wedge\mathrm{d}\bar z^i.
\]
Here $\lambda_i=\cot\phi_i$, with $\phi_i>0$ and
$\sum_i\phi_i=\theta<\pi$.
We use the notation $[n]:=\{1,\ldots,n\}$ and let
$\Lambda^{p,q}$ denote the space of complex $(p,q)$-forms
at the chosen point, with $\Lambda^{p,q}=0$ if either index
lies outside $\{0,\ldots,n\}$.
For $0\le j<n$, evaluating
$P_j$ from \eqref{eq:poly} on $(\kappa,\omega)$ and expanding gives
\begin{equation}\label{sp:eq:P}
 c_j(\alpha,\omega)=(\sqrt{-1})^{\,j^2}\binom{n-1}{j}j!\sum_{|H|=j}
 \frac{\sin(\theta-\sum_{i\in H}\phi_i)}
 {\sin\theta\prod_{i\in H}\sin\phi_i}\,
 \mathrm{d}z^H\wedge\mathrm{d}\bar z^H.
\end{equation}
All these coefficients are positive. For fixed $0<\theta<\pi$ and
$N\ge0$, we write
\[
 F_N(x):=\operatorname{Re}\prod_{i=1}^N(x_i+\sqrt{-1})
 -\cot\theta\operatorname{Im}\prod_{i=1}^N(x_i+\sqrt{-1}).
\]
For $\lambda=(\cot\phi_i)_{i=1}^N$ with $0<\phi_i<\pi$, substituting
$\lambda_i+\sqrt{-1}=e^{\sqrt{-1}\phi_i}/\sin\phi_i$ gives
\begin{equation}\label{sp:eq:phase-polynomial}
 F_N(\lambda)
 =\frac{\cos(\sum_i\phi_i)-\cot\theta\sin(\sum_i\phi_i)}
 {\prod_i\sin\phi_i}
 =\frac{\sin(\theta-\sum_i\phi_i)}
 {\sin\theta\prod_i\sin\phi_i}.
\end{equation}
The first and second derivatives of $F_N$ give the
coefficients of the primitive condition and the pairing below.
For the proof of Lemma~\ref{sp:lem:phase-HR}, we follow Yuan's techniques in the proof of
convexity for supercritical special Lagrangian level sets
\cite[Lemma~2.1]{Yuan}.
This argument also appears in the proof of
\cite[Proposition~4.1]{FYZZ}.

\begin{lemma}\label{sp:lem:phase-HR}
Let $N\ge2$ and $\lambda_i=\cot\phi_i$, where $\phi_i>0$ and
$\sum_i\phi_i\le\theta<\pi$.
The Hessian matrix $H:=(\partial_i\partial_jF_N(\lambda))$,
with $H(z,w):=\sum_{i,j}H_{ij}z_i\bar w_j$, has signature $(1,N-1)$
as a Hermitian form.
The following hold.
\begin{enumerate}[label=(\roman*)]
\item The form $H$ is negative definite on
$\{z\in\C^N:\sum_i\partial_iF_N(\lambda)z_i=0\}$.
\item For every $x\in\R_{>0}^N$, the form $H$ is negative definite on
$\{z\in\C^N:H(z,x)=0\}$.
\end{enumerate}
\end{lemma}
\begin{proof}
We first prove (i). Since the coefficients $\partial_iF_N(\lambda)$ are real,
the real and imaginary parts of a vector satisfying the primitive condition
in (i) both satisfy the same condition. Writing $z=\operatorname{Re}z+\sqrt{-1}\operatorname{Im}z$,
we expand the Hermitian form as
\[
\begin{aligned}
 H(z,z)
 &=H(\operatorname{Re}z,\operatorname{Re}z)
   +H(\operatorname{Im}z,\operatorname{Im}z)+\sqrt{-1}\bigl(H(\operatorname{Im}z,\operatorname{Re}z)
   -H(\operatorname{Re}z,\operatorname{Im}z)\bigr)\\
 &=H(\operatorname{Re}z,\operatorname{Re}z)
   +H(\operatorname{Im}z,\operatorname{Im}z),
\end{aligned}
\]
where the cross terms cancel because the matrix $H$ is real symmetric.
For $z\ne0$, at least one of $\operatorname{Re}z$ and $\operatorname{Im}z$ is nonzero. It therefore suffices to prove (i) for real vectors.

Direct differentiation gives
\[
\begin{aligned}
 \partial_iF_N(x)
 &=\operatorname{Re}\prod_{k\ne i}(x_k+\sqrt{-1})
 -\cot\theta\operatorname{Im}\prod_{k\ne i}(x_k+\sqrt{-1}),\\
 \partial_i\partial_jF_N(x)
 &=\operatorname{Re}\prod_{k\ne i,j}(x_k+\sqrt{-1})
 -\cot\theta\operatorname{Im}\prod_{k\ne i,j}(x_k+\sqrt{-1})
 \qquad(i\ne j).
\end{aligned}
\]
Since $\partial_iF_N(x)$ is independent of $x_i$, we have
$H_{ii}=\partial_i^2F_N(\lambda)=0$.
We put $\delta:=\theta-\sum_i\phi_i\ge0$.
Substituting $\lambda_k+\sqrt{-1}=e^{\sqrt{-1}\phi_k}/\sin\phi_k$ gives
\begin{equation}\label{eq:phase-derivatives}
\begin{aligned}
 \partial_iF_N(\lambda)
 &=\frac{\sin(\delta+\phi_i)\sin\phi_i}
 {\sin\theta\prod_k\sin\phi_k}>0,\\
 H_{ij}
 &=\frac{\sin(\delta+\phi_i+\phi_j)\sin\phi_i\sin\phi_j}
 {\sin\theta\prod_k\sin\phi_k}>0
 \qquad(i\ne j),
\end{aligned}
\end{equation}
since $0<\delta+\phi_i\le\delta+\phi_i+\phi_j\le\theta<\pi$.

We take $0\ne v\in\R^N$ satisfying the primitive condition in (i),
that is, $\sum_i\partial_iF_N(\lambda)v_i=0$, and set $w_i:=v_i\sin^2\phi_i$.
By \eqref{eq:phase-derivatives} and the identity
$\sin(\delta+\phi_i)=\sin\phi_i(\cos\delta+\lambda_i\sin\delta)$,
the primitive condition becomes
$\cos\delta\sum_iw_i+\sin\delta\sum_i\lambda_iw_i=0$.
By \eqref{eq:phase-derivatives},
\[
\begin{aligned}
 \sin\theta\left(\prod_i\sin\phi_i\right)H(v,v)
 &=\sum\nolimits_{i\ne j}
 \frac{\sin(\delta+\phi_i+\phi_j)}{\sin\phi_i\sin\phi_j}w_iw_j\\
 &=\sin\delta\Bigl[
 \bigl(\sum\nolimits_i\lambda_iw_i\bigr)^2
 -\bigl(\sum\nolimits_iw_i\bigr)^2\Bigr]\\
 &\quad+2\cos\delta\bigl(\sum\nolimits_i\lambda_iw_i\bigr)
 \bigl(\sum\nolimits_iw_i\bigr)
 -\sum\nolimits_i\frac{\sin(\delta+2\phi_i)}{\sin^2\phi_i}w_i^2.
\end{aligned}
\]

\noindent We relabel the indices so that $\lambda_1\le\cdots\le\lambda_N$,
or equivalently $\phi_1\ge\cdots\ge\phi_N>0$.

For $\delta=0$, the primitive condition gives $\sum_iw_i=0$.
Substituting $\delta=0$ into the preceding expansion gives
\[
\begin{aligned}
 -\sin\theta\left(\prod_i\sin\phi_i\right)H(v,v)
 &=-2\left(\sum_i\lambda_iw_i\right)
       \underbrace{\left(\sum_iw_i\right)}_{=0}
   +\sum_i\frac{\sin(2\phi_i)}{\sin^2\phi_i}w_i^2\\
 &=2\sum_i\cot\phi_i\,w_i^2\\ 
  &=2\sum_i\lambda_iw_i^2.
\end{aligned}
\]
Since $2\phi_2\le\phi_1+\phi_2\le\sum_i\phi_i<\pi$,
we have $\phi_i<\pi/2$ for $i\ge2$, hence
$0<\lambda_2\le\cdots\le\lambda_N$.
If $\lambda_1\ge0$, the conditions $w\ne0$ and $\sum_iw_i=0$
imply $(w_2,\ldots,w_N)\ne0$, so $\sum_i\lambda_iw_i^2>0$.
Otherwise, $\lambda_1<0$.
The bound $\sum_{i>1}\phi_i<\pi-\phi_1<\pi/2$ gives
$\sum_{i>1}\lambda_i^{-1}\le\tan(\sum_{i>1}\phi_i)<-\lambda_1^{-1}$.
Weighted Cauchy--Schwarz now yields
\[
 \sum_i\lambda_iw_i^2
 \ge\left(1+\lambda_1\sum_{i>1}\frac1{\lambda_i}\right)
       \sum_{i>1}\lambda_iw_i^2>0.
\]

For $\delta>0$, the primitive condition gives
$\sum_i\lambda_iw_i=-\cot\delta\sum_iw_i$.
Substituting this equality into the quadratic expansion gives
\[
 -\sin\theta\left(\prod_i\sin\phi_i\right)H(v,v)
 =\sum_i\frac{\sin(\delta+2\phi_i)}
                  {\sin^2\phi_i}w_i^2
   +\frac{(\sum_iw_i)^2}{\sin\delta}.
\]
For $i\ge2$, the ordering gives
\[
 0<\delta+2\phi_i\le\delta+\phi_1+\phi_i\le\theta<\pi,
\]
so the coefficients of $w_i^2$ are positive.
We now consider the coefficient of $w_1^2$.
Suppose first that $\delta+2\phi_1\le\pi$.
Then all diagonal coefficients are nonnegative.
If the right-hand side were zero, we would have $w_i=0$ for $i\ge2$
and $\sum_iw_i=0$, hence $w=0$.
Thus the right-hand side is positive for $w\ne0$ in this case.

Suppose now that $\delta+2\phi_1>\pi$.
Since $\delta+\phi_1<\pi$, we have $\delta+2\phi_1<2\pi$,
so the coefficient of $w_1^2$ is negative.
We must show that the sum of the remaining positive terms is greater than
the absolute value of this negative term.
We put
$\sigma:=\sum_{i>1}\phi_i<\pi-\delta-\phi_1<(\pi-\delta)/2$.
On $0<u<(\pi-\delta)/2$, we set
\[
 g(u):=\frac{\sin^2u}{\sin(\delta+2u)},\qquad
 \frac{g(u)}{\tan u}
 =\frac1{2\bigl(\cos\delta+\sin\delta\cot(2u)\bigr)}.
\]
Both $\tan u$ and $g(u)/\tan u$ are positive and strictly increasing.
In particular, $g$ is strictly increasing.
Since $\phi_i\le\sigma$ for $i>1$ and
$\sum_{i>1}\tan\phi_i\le\tan\sigma$, we have
\[
 \sum_{i>1}g(\phi_i)
 =\sum_{i>1}\frac{g(\phi_i)}{\tan\phi_i}\tan\phi_i
 \le\frac{g(\sigma)}{\tan\sigma}\sum_{i>1}\tan\phi_i
 \le g(\sigma).
\]
It follows that
\[
\begin{aligned}
 0<\sin\delta+\sum\nolimits_{i>1}g(\phi_i)
 &\le\sin\delta+g(\sigma)\\[-3pt]
 &<\sin\delta+g(\pi-\delta-\phi_1)\\
 &=-\frac{\sin^2\phi_1}{\sin(\delta+2\phi_1)}.
\end{aligned}
\]
If $w_1\ne0$, weighted Cauchy--Schwarz applied to $w_i$ for $i>1$ and
$-\sum_iw_i$, whose sum is $-w_1$, gives
\[
\begin{aligned}
 \sum_{i>1}\frac{w_i^2}{g(\phi_i)}
 +\frac{(\sum_iw_i)^2}{\sin\delta}
 &\ge\frac{w_1^2}{\sin\delta+\sum_{i>1}g(\phi_i)}\\
 &>-\frac{\sin(\delta+2\phi_1)}
 {\sin^2\phi_1}w_1^2.
\end{aligned}
\]
When $w_1=0$, we have $(w_2,\ldots,w_N)\ne0$, so
$\sum_{i>1}w_i^2/g(\phi_i)>0$.
Thus the right-hand side is positive for $w\ne0$ in this case as well,
and hence $H(v,v)<0$. This proves (i).

By (i), the matrix $H$ has at least $N-1$ negative eigenvalues.
Since $H_{ii}=0$, we have $\operatorname{tr}H=0$, so the remaining
eigenvalue is positive. Thus the matrix $H$ has signature $(1,N-1)$.
We now prove (ii). We fix $x\in\R_{>0}^N$.
By \eqref{eq:phase-derivatives}, $H_{ij}>0$ for $i\ne j$.
Since every $x_i>0$ and $N\ge2$,
$H(x,x)=2\sum_{i<j}H_{ij}x_ix_j>0$.
Thus the form $H$ is positive definite on $\C x$.
In the $H$-orthogonal decomposition
\[
 \C^N=\C x\oplus\{z\in\C^N:H(z,x)=0\},
\]
the restriction of $H$ to the second summand therefore has
signature $(0,N-1)$ and is negative definite.
\end{proof}

We next apply this Hessian inequality to the pairing on differential forms.
For $u,v\in\Lambda^{p,q}$, we define
\[
 q_{\mathcal R}^{p,q}(u,v)(\sqrt{-1})^{\,n^2}
 \mathrm{d}z^{[n]}\wedge\mathrm{d}\bar z^{[n]}
 :=(\sqrt{-1})^{\,p-q}(-1)^{(p+q)(p+q-1)/2}
 u\wedge\bar v\wedge c_{n-p-q}(\alpha,\omega).
\]
We omit the superscript when the bidegree is fixed.

\begin{proposition}\label{sp:prop:equality}
Assume $1\le p+q\le n$ and $\min\{p,q\}\le1$. We set
$m:=n-p-q+1$.
The following hold at every point.
\begin{enumerate}[label=(\roman*)]
\item The map $c_{m-1}(\alpha,\omega)\wedge-\colon\Lambda^{p,q}\to\Lambda^{n-q,n-p}$
is an isomorphism.
\item Both $c_m(\alpha,\omega)\wedge-$ and $c_{m-1}(\alpha,\omega)\wedge\kappa\wedge-$
from $\Lambda^{p,q}$ to $\Lambda^{n-q+1,n-p+1}$ are surjective.
\item The form $q_{\mathcal R}$ is positive definite on each of
$\ker(c_m(\alpha,\omega)\wedge-)$ and $\ker(c_{m-1}(\alpha,\omega)\wedge\kappa\wedge-)$.
\end{enumerate}
\end{proposition}
\begin{proof}
If $q=0$, \eqref{sp:eq:P} gives
\[
 c_{n-p}(\alpha,\omega)=(\sqrt{-1})^{\,(n-p)^2}
 \sum_{|L|=n-p}\gamma_L\,\mathrm{d}z^L\wedge\mathrm{d}\bar z^L,
\]
with $\gamma_L>0$. For $0\ne u=\sum_{|I|=p}u_I\,\mathrm{d}z^I$,
we obtain
\[
 q_{\mathcal R}(u,u)=\sum_{\substack{I\sqcup L=[n]\\|I|=p}}
 \gamma_L|u_I|^2>0.
\]
The case $p=0$ follows by conjugation. In these cases, both maps
in (ii) have zero target, so (ii) and (iii) follow. The perfect
exterior-product pairing gives (i).
By conjugation, it remains to consider $q=1$ and $1\le p\le n-1$.
Then $m=n-p$.

For an index set $L\subset[n]$, write $L^c:=[n]\setminus L$ and let
$\lambda_L:=(\lambda_i)_{i\in L}$ denote the tuple of eigenvalues of
$\alpha$ with respect to $\omega$ indexed by $L$, in increasing index order.
For $u\in\Lambda^{p,1}$, we may decompose it as
\[
 u=\sum_{\substack{|I|=p\\i\in I^c}}
 a_{I,i}\,\mathrm{d}z^I\wedge\mathrm{d}\bar z^i
 +\sum_{\substack{|J|=p-1\\i\in J^c}}
 b_{J,i}\,\mathrm{d}z^J\wedge
 \mathrm{d}z^i\wedge\mathrm{d}\bar z^i.
\]
The forms $\eta_{I,i}:=\mathrm{d}z^I\wedge\mathrm{d}\bar z^i$ and
$\zeta_{J,i}:=\mathrm{d}z^J\wedge\mathrm{d}z^i\wedge\mathrm{d}\bar z^i$,
with the indices as above, form a basis of $\Lambda^{p,1}$.
We compute $q_{\mathcal R}$ with respect to this basis.
Using \eqref{sp:eq:phase-polynomial} in \eqref{sp:eq:P} gives
\begin{equation}\label{sp:eq:chern-coefficients}
 \frac{c_{m-1}(\alpha,\omega)}{\binom{n-1}{m-1}(m-1)!}
 =(\sqrt{-1})^{\,(m-1)^2}\sum_{|L|=m-1}F_{m-1}(\lambda_L)
 \,\mathrm{d}z^L\wedge\mathrm{d}\bar z^L.
\end{equation}
Using \eqref{sp:eq:chern-coefficients}, we obtain
\[
\begin{aligned}
 \eta_{I,i}\wedge\overline{\eta_{I,j}}\wedge c_{m-1}(\alpha,\omega)
 &=\begin{cases}
 -\mathrm{d}z^{I\cup\{i\}}\wedge
 \mathrm{d}\bar z^{I\cup\{i\}}\wedge c_{m-1}(\alpha,\omega),&i=j,\\
 0,&i\ne j,
 \end{cases}\\
 \zeta_{J,i}\wedge\overline{\zeta_{J,j}}\wedge c_{m-1}(\alpha,\omega)
 &=\begin{cases}
 \mathrm{d}z^{J\cup\{i,j\}}\wedge
 \mathrm{d}\bar z^{J\cup\{i,j\}}\wedge c_{m-1}(\alpha,\omega),&i\ne j,\\
 0,&i=j.
 \end{cases}
\end{aligned}
\]
The definition of $q_{\mathcal R}$ then gives
\begin{equation}\label{sp:eq:monomial-pairings}
\begin{aligned}
 \frac{q_{\mathcal R}(\eta_{I,i},\eta_{I',j})}{\binom{n-1}{m-1}(m-1)!}
 &=\begin{cases}
 F_{m-1}(\lambda_{(I\cup\{i\})^c}),&I=I',\ i=j,\\
 0,&\text{otherwise},
 \end{cases}\\
 q_{\mathcal R}(\eta_{I,i},\zeta_{J,j})&=0,\\
 \frac{q_{\mathcal R}(\zeta_{J,i},\zeta_{J',j})}{\binom{n-1}{m-1}(m-1)!}
 &=\begin{cases}
 -F_{m-1}(\lambda_{(J\cup\{i,j\})^c}),&J=J',\ i\ne j,\\
 0,&\text{otherwise}.
 \end{cases}
\end{aligned}
\end{equation}
For $|J|=p-1$, we define
$F_J(\lambda):=F_{m+1}(\lambda_{J^c})$ and
$H_J:=\bigl(\partial_i\partial_jF_J(\lambda)\bigr)_{i,j\in J^c}$.
Then
\begin{equation}\label{sp:eq:block-derivatives}
\begin{aligned}
 \partial_iF_J(\lambda)
 &=F_m(\lambda_L),
 &&J\sqcup\{i\}\sqcup L=[n],\\
 (H_J)_{ij}
 &=F_{m-1}(\lambda_L),
 &&J\sqcup\{i,j\}\sqcup L=[n],\quad i\ne j,\\
 (H_J)_{ii}&=0,
 &&i\in J^c.
\end{aligned}
\end{equation}
With $b_J:=(b_{J,i})_{i\in J^c}$, combining
\eqref{sp:eq:monomial-pairings} and
\eqref{sp:eq:block-derivatives} gives
\begin{equation}\label{sp:eq:direct-HR}
 \frac{q_{\mathcal R}(u,u)}{\binom{n-1}{m-1}(m-1)!}
 =\sum_{\substack{I\sqcup\{i\}\sqcup L=[n]\\|I|=p}}
 F_{m-1}(\lambda_L)|a_{I,i}|^2
 -\sum_{|J|=p-1}H_J(b_J,b_J).
\end{equation}
By \eqref{sp:eq:P}, $F_{m-1}(\lambda_L)>0$.
Since $|J^c|=m+1\ge2$ and $\sum_{i\in J^c}\phi_i\le\theta<\pi$,
Lemma~\ref{sp:lem:phase-HR} applies to $F_{m+1}$ at $\lambda_{J^c}$
and shows that $H_J$ has signature $(1,m)$.
Hence each $H_J$ is nondegenerate, and so is the pairing $q_{\mathcal R}$.
By the definition of $q_{\mathcal R}$,
\[
 q_{\mathcal R}(u,v)(\sqrt{-1})^{\,n^2}
 \mathrm{d}z^{[n]}\wedge\mathrm{d}\bar z^{[n]}
 =(\sqrt{-1})^{\,p-1}(-1)^{p(p+1)/2}
 \bigl(c_{m-1}(\alpha,\omega)\wedge u\bigr)\wedge\bar v.
\]
Consequently, if $c_{m-1}(\alpha,\omega)\wedge u=0$, then $q_{\mathcal R}(u,v)=0$ for
every $v\in\Lambda^{p,1}$, and nondegeneracy forces $u=0$.
The source and target have equal dimension, so this proves (i).

We now prove (ii) and (iii).
Since $\kappa=\sqrt{-1}\sum_i(\lambda_i-\cot\theta)
\,\mathrm{d}z^i\wedge\mathrm{d}\bar z^i$,
\eqref{sp:eq:chern-coefficients} gives
\[
 \frac{c_{m-1}(\alpha,\omega)\wedge\kappa}{\binom{n-1}{m-1}(m-1)!}
 =(\sqrt{-1})^{\,m^2}\sum_{|L|=m}
 \left(\sum_{j\in L}(\lambda_j-\cot\theta)
 F_{m-1}(\lambda_{L\setminus\{j\}})\right)
 \mathrm{d}z^L\wedge\mathrm{d}\bar z^L.
\]
We compute the wedge-product maps in (ii) with respect to the basis
$\eta_{I,i},\zeta_{J,i}$ of $\Lambda^{p,1}$ and the basis
\[
 \tau_J:=(\sqrt{-1})^{\,m+1}\mathrm{d}z^J\wedge
 \bigwedge_{k\in J^c}(\mathrm{d}z^k\wedge\mathrm{d}\bar z^k),
 \qquad |J|=p-1,
\]
of $\Lambda^{n,n-p+1}$.
Using the expansion above, \eqref{sp:eq:P}, and
\eqref{sp:eq:block-derivatives}, a direct computation on these basis
forms gives
\begin{equation}\label{sp:eq:diagonal-wedges}
\begin{aligned}
 c_m(\alpha,\omega)\wedge\eta_{I,i}
 &=0,\\
 c_{m-1}(\alpha,\omega)\wedge\kappa\wedge\eta_{I,i}
 &=0,\\
 \frac{c_m(\alpha,\omega)\wedge\zeta_{J,i}}{\binom{n-1}{m}m!}
 &=-\sqrt{-1}\,\partial_iF_J(\lambda)\tau_J,\\
 \frac{c_{m-1}(\alpha,\omega)\wedge\kappa\wedge\zeta_{J,i}}
 {\binom{n-1}{m-1}(m-1)!}
 &=-\sqrt{-1}\left(\sum_{j\in J^c}(H_J)_{ij}
 (\lambda_j-\cot\theta)\right)\tau_J.
\end{aligned}
\end{equation}
To prove (ii), fix $J$ with $|J|=p-1$ and $i\in J^c$.
By \eqref{eq:phase-derivatives}, we have $\partial_iF_J(\lambda)>0$,
$(H_J)_{ij}>0$ for $j\ne i$, and $(H_J)_{ii}=0$.
Moreover, $\phi_j<\theta$ gives $\lambda_j-\cot\theta>0$.
Since $|J^c|=m+1\ge2$, we obtain
\[
 \sum_{j\in J^c}(H_J)_{ij}(\lambda_j-\cot\theta)
 =\sum_{\substack{j\in J^c\\j\ne i}}
 (H_J)_{ij}(\lambda_j-\cot\theta)>0.
\]
The last equality in \eqref{sp:eq:diagonal-wedges} therefore shows that
the map $c_{m-1}(\alpha,\omega)\wedge\kappa\wedge-$ sends
$\zeta_{J,i}$ to a nonzero multiple of $\tau_J$.
The third equality in \eqref{sp:eq:diagonal-wedges} and
$\partial_iF_J(\lambda)>0$ give the same
conclusion for $c_m(\alpha,\omega)\wedge-$.
Every target basis vector $\tau_J$ therefore lies in the image of
each map. This proves (ii).

For (iii), we first express the two primitive conditions in terms of
the vectors $b_J$. By linearity, \eqref{sp:eq:diagonal-wedges} gives
\begin{equation}\label{sp:eq:direct-primitive}
\begin{aligned}
 \frac{c_m(\alpha,\omega)\wedge u}{\binom{n-1}{m}m!}
 &=-\sqrt{-1}\sum_{|J|=p-1}
 \left(\sum_{i\in J^c}\partial_iF_J(\lambda)b_{J,i}\right)\tau_J,\\
 \frac{c_{m-1}(\alpha,\omega)\wedge\kappa\wedge u}
 {\binom{n-1}{m-1}(m-1)!}
 &=-\sqrt{-1}\sum_{|J|=p-1}
 H_J\bigl(b_J,(\lambda_i-\cot\theta)_{i\in J^c}\bigr)\tau_J.
\end{aligned}
\end{equation}
Suppose first that $c_m(\alpha,\omega)\wedge u=0$.
The forms $\tau_J$ are linearly independent, so the first identity gives
\[
 \sum_{i\in J^c}\partial_iF_J(\lambda)b_{J,i}=0
 \qquad\text{for every }J.
\]
Lemma~\ref{sp:lem:phase-HR}(i) therefore gives
$H_J(b_J,b_J)\le0$, with equality only when $b_J=0$.

Now suppose that $c_{m-1}(\alpha,\omega)\wedge\kappa\wedge u=0$.
The second identity in \eqref{sp:eq:direct-primitive} and the same
linear independence give
\[
 H_J\bigl(b_J,(\lambda_i-\cot\theta)_{i\in J^c}\bigr)=0
 \qquad\text{for every }J.
\]
All entries of $(\lambda_i-\cot\theta)_{i\in J^c}$ are positive,
so Lemma~\ref{sp:lem:phase-HR}(ii) again gives
$H_J(b_J,b_J)\le0$, with equality only when $b_J=0$.

In either kernel, every term on the right-hand side of
\eqref{sp:eq:direct-HR} is nonnegative.
For $u\ne0$, its basis expansion has either a nonzero coefficient
$a_{I,i}$ or a nonzero vector $b_J$.
In the first case, the corresponding term
$F_{m-1}(\lambda_L)|a_{I,i}|^2$ is positive;
in the second, $-H_J(b_J,b_J)>0$.
Hence $q_{\mathcal R}(u,u)>0$ for every nonzero $u$ in either kernel,
proving (iii).
\end{proof}

\subsection{K\"ahler package on cohomology}
To obtain the cohomological relations, we use the construction of
primitive representatives by Dinh--Nguy\^en
\cite[Propositions~3.1--3.2]{DN} for a pair of closed forms.
We write $\Lambda^{p,q}T^*X$ for the smooth bundle of complex
$(p,q)$-forms, whose fibre at $x$ is $\Lambda^{p,q}T_x^*X$.
We fix a background K\"ahler form and denote its volume form by
$\mathrm{vol}$.
For every bidegree $(p,q)$, we equip $\Lambda^{p,q}T^*X$ with the
Hermitian metric $h_0$ induced by the background K\"ahler form.
We denote the associated pointwise and $L^2$ norms by $\|\cdot\|$
and $\|\cdot\|_{L^2}$, respectively.
For $(p,q)$-forms $u,v$ and a real $(n-p-q,n-p-q)$-form $\Omega$, we set
\[
\begin{aligned}
 q_\Omega^{p,q}(u,v)\,\mathrm{vol}
 &:=(\sqrt{-1})^{\,p-q}(-1)^{(p+q)(p+q-1)/2}
 u\wedge\bar v\wedge\Omega,\\
 Q_\Omega^{p,q}(u,v)
 &:=\int_Xq_\Omega^{p,q}(u,v)\,\mathrm{vol}.
\end{aligned}
\]
When $\Omega$ is closed, Stokes' theorem shows that the restriction
of $Q_\Omega^{p,q}$ to closed forms descends to the cohomological
pairing
\[
 Q_{[\Omega]}^{p,q}([u],[v]):=Q_\Omega^{p,q}(u,v).
\]
We use the convention
$\mathrm{d}\mathrm{d}^c:=\sqrt{-1}\partial\bar\partial$.

\begin{proposition}\label{sp:prop:adjacent-global}
Let $X$ be a compact K\"ahler $n$-fold, and let $\Psi$, $\Omega$ be
smooth real closed forms of bidegrees $(m,m)$, $(m-1,m-1)$,
where $1\le m\le n$.
We fix $p,q\ge0$ with $p+q=n-m+1$.
Assume that the following hold at every point $x\in X$.
\begin{enumerate}[label=(\roman*)]
\item The map $\Psi_x\wedge-\colon \Lambda^{p,q}T_x^*X\to
\Lambda^{p+m,q+m}T_x^*X$ is surjective.
\item The form $q_\Omega^{p,q}$ is positive definite on $\ker(\Psi_x\wedge-)$.
\end{enumerate}
Then every class in $\ker([\Psi]\cup-)$ has a smooth closed
representative $u$ with $\Psi\wedge u=0$.
In particular, the form $Q_{[\Omega]}^{p,q}$ is positive definite on this cohomological kernel.
\end{proposition}
\begin{proof}
If $p=0$ or $q=0$, every closed representative $u$ is primitive by
degree, and $Q_\Omega(u,u)$ is its squared $L^2$ norm induced by
$q_\Omega$.
We therefore assume $p,q\ge1$.

Fix a smooth closed $(p,q)$-form $f$ with $[\Psi]\cup[f]=0$.
We seek a primitive representative $u=f-\mathrm{d}\mathrm{d}^cv$
by solving
\[
 \Psi\wedge\mathrm{d}\mathrm{d}^cv=\Psi\wedge f.
\]
For every smooth $(q-1,p-1)$-form $\gamma$, testing this equation
and integrating by parts gives the required identity
\[
 \int_X\bigl(\Psi\wedge\mathrm{d}\mathrm{d}^c\gamma\bigr)\wedge v
 =\int_X\gamma\wedge\Psi\wedge f.
\]
Let $\mathcal E$ be the linear subspace of
$L^2(X,\Lambda^{n-p+1,n-q+1}T^*X)$ consisting of the forms
$\Psi\wedge\mathrm{d}\mathrm{d}^c\gamma$, where $\gamma$ is a smooth
$(q-1,p-1)$-form. We consider the functional
$\ell\bigl(\Psi\wedge\mathrm{d}\mathrm{d}^c\gamma\bigr)
:=\int_X\gamma\wedge\Psi\wedge f$ on $\mathcal E$.
To show that it is well defined and bounded, we will prove that
there is a constant $C_f>0$ such that
\[
 \left|\int_X\gamma\wedge\Psi\wedge f\right|
 \le C_f\|\Psi\wedge\mathrm{d}\mathrm{d}^c\gamma\|_{L^2}
\]
for every smooth $(q-1,p-1)$-form $\gamma$.
Indeed, if $\Psi\wedge\mathrm{d}\mathrm{d}^c\gamma_1
=\Psi\wedge\mathrm{d}\mathrm{d}^c\gamma_2$, applying this bound to
$\gamma_1-\gamma_2$ gives
$\int_X(\gamma_1-\gamma_2)\wedge\Psi\wedge f=0$.
Thus the two choices give the same value of $\ell$, so $\ell$ is
well defined. The same bound gives its continuity.
Once it is proved, Hahn--Banach extends $\ell$ to a bounded linear
functional $\widetilde\ell$ on $L^2(X,\Lambda^{n-p+1,n-q+1}T^*X)$.
By $L^2$-duality, as in \cite[Proposition~3.1]{DN}, this extension
has the form $\widetilde\ell(\eta)=\int_X\eta\wedge v$ for a
$(p-1,q-1)$-form $v$ in $L^2$.
Its agreement with $\ell$ on $\mathcal E$ is exactly the tested
identity above, so $v$ solves the desired equation weakly.
We will then prove that $u$ is smooth by elliptic regularity.

We now establish the required estimate.
For the smooth bundle map
$\Psi\wedge-\colon\allowbreak\Lambda^{p,q}T^*X\to\allowbreak
\Lambda^{p+m,q+m}T^*X$,
fibrewise surjectivity follows from (i), so its rank is constant.
Its kernel $P$ is therefore a smooth subbundle of $\Lambda^{p,q}T^*X$.
By (ii), the restriction $q_\Omega|_P$ is positive definite, so
\[
 \Lambda^{p,q}T^*X=P\oplus P^{\perp_{q_\Omega}}.
\]
On the unit sphere bundle of $P^{\perp_{q_\Omega}}$, the continuous function
$z\mapsto\|\Psi\wedge z\|$ is strictly positive since
$P=\ker(\Psi\wedge-)$.
This sphere bundle is compact.
Consequently, the function attains a positive minimum $\delta$,
and the continuous function $z\mapsto|q_\Omega(z,z)|$ attains a
finite maximum $M$ on the same compact sphere bundle.
For nonzero $z\in P_x^{\perp_{q_\Omega}}$, we apply these bounds to
$z/\|z\|$ to obtain
\begin{equation}\label{sp:eq:complement-bounds}
 \|\Psi\wedge z\|\ge\delta\|z\|,
 \qquad |q_\Omega(z,z)|\le M\|z\|^2.
\end{equation}
On the other hand, (ii) and compactness of the unit sphere bundle
of $P$ give a uniform constant $a>0$ such that, for every $w\in P_x$,
\begin{equation}\label{sp:eq:primitive-lower-bound}
 q_\Omega(w,w)\ge a\|w\|^2.
\end{equation}
We now fix $x\in X$ and write any $\zeta\in\Lambda^{p,q}T_x^*X$ as
$\zeta=w+z$, with $w\in P_x$ and $z\in P_x^{\perp_{q_\Omega}}$.
The $q_\Omega$-orthogonality of the two summands and
$P_x=\ker(\Psi_x\wedge-)$ give
\begin{equation}\label{sp:eq:primitive-splitting}
 q_\Omega(\zeta,\zeta)=q_\Omega(w,w)+q_\Omega(z,z),
 \qquad \Psi\wedge \zeta=\Psi\wedge z.
\end{equation}
It follows that
\begin{equation}\label{sp:eq:pointwise-estimate}
\begin{aligned}
 \|\zeta\|^2
 &\le2\|w\|^2+2\|z\|^2\\
 &\stackrel{(1)}{\le}\frac{2}{a}q_\Omega(w,w)+2\|z\|^2\\
 &\stackrel{(2)}{=}\frac{2}{a}q_\Omega(\zeta,\zeta)-\frac{2}{a}q_\Omega(z,z)+2\|z\|^2\\
 &\le\frac{2}{a}q_\Omega(\zeta,\zeta)
      +2\left(1+\frac{M}{a}\right)\|z\|^2\\
 &\le C_1q_\Omega(\zeta,\zeta)+C_2\|\Psi\wedge \zeta\|^2,
\end{aligned}
\end{equation}
where (1) follows from \eqref{sp:eq:primitive-lower-bound}
and (2) from
\eqref{sp:eq:primitive-splitting}.
Here $C_1:=2/a$ and $C_2:=2(1+M/a)/\delta^2$.
Since $\Psi$ and $\Omega$ are real, conjugation preserves
the norms and satisfies
$q_\Omega^{q,p}(\xi,\xi)=q_\Omega^{p,q}(\bar\xi,\bar\xi)$.
Applying \eqref{sp:eq:pointwise-estimate} to $\bar\xi$ therefore gives the
same estimate for a $(q,p)$-form $\xi$.

To obtain an $L^2$ estimate, we choose a smooth $(q-1,p-1)$-form $\gamma$.
Since $\mathrm{d}\Omega=0$, Stokes' theorem gives
\[
 Q_\Omega^{q,p}(\mathrm{d}\mathrm{d}^c\gamma,\mathrm{d}\mathrm{d}^c\gamma)
 = (\sqrt{-1})^{\,q-p}(-1)^{(p+q)(p+q-1)/2}
 \int_X\mathrm{d}\mathrm{d}^c\gamma\wedge
 \overline{\mathrm{d}\mathrm{d}^c\gamma}\wedge\Omega=0.
\]
By \eqref{sp:eq:pointwise-estimate},
\[
 \|\mathrm{d}\mathrm{d}^c\gamma\|_{L^2}^2
 \le C_1Q_\Omega^{q,p}(\mathrm{d}\mathrm{d}^c\gamma,\mathrm{d}\mathrm{d}^c\gamma)
      +C_2\|\Psi\wedge\mathrm{d}\mathrm{d}^c\gamma\|_{L^2}^2
 =C_2\|\Psi\wedge\mathrm{d}\mathrm{d}^c\gamma\|_{L^2}^2,
\]
and thus
\[
 \|\mathrm{d}\mathrm{d}^c\gamma\|_{L^2}
 \le C\|\Psi\wedge\mathrm{d}\mathrm{d}^c\gamma\|_{L^2}.
\]
The $\mathrm{d}\mathrm{d}^c$-lemma gives a smooth $(n-q,n-p)$-form $g$
such that $\mathrm{d}\mathrm{d}^c g=\Psi\wedge f$.
For every smooth $(q-1,p-1)$-form $\gamma$, Stokes' theorem gives
\[
 \left|\int_X\gamma\wedge\Psi\wedge f\right|
 =\left|\int_X\mathrm{d}\mathrm{d}^c\gamma\wedge g\right|
 \le C\|g\|_{L^2}\|\Psi\wedge\mathrm{d}\mathrm{d}^c\gamma\|_{L^2}.
\]
This proves the required bound for $\ell$.
The construction above therefore gives a $(p-1,q-1)$-form $v$ in
$L^2$ satisfying $\Psi\wedge\mathrm{d}\mathrm{d}^cv=\Psi\wedge f$
in the weak sense.

We set $u:=f-\mathrm{d}\mathrm{d}^c v$, with derivatives understood
in the weak sense. Then $\partial u=\bar\partial u=0$ as $f$ is
$\mathrm{d}$-closed.
The weak equation above gives $\Psi\wedge u=0$, so $u$ is a
distributional section of $P$.
To prove that $u$ is smooth, we extend $q_\Omega|_P$ to a positive
Hermitian metric $h$ on $\Lambda^{p,q}T^*X$.
For $z=z_P+z_\perp$ and $w=w_P+w_\perp$ in
$P_x\oplus P_x^{\perp_{q_\Omega}}$, using the induced metric $h_0$, we define
\[
 h(z,w):=q_\Omega(z_P,w_P)+h_0(z_\perp,w_\perp).
\]
For $w\in P_x$, this gives
\begin{equation}\label{sp:eq:primitive-metric}
 h(z,w)=q_\Omega(z_P,w)=q_\Omega(z,w).
\end{equation}

\noindent
Let $\partial_h^*$ denote the formal adjoint of $\partial$, using
$h$ in bidegree $(p,q)$ and $h_0$ in every other bidegree.
Since $u$ is a distributional section of $P$,
\eqref{sp:eq:primitive-metric} and distributional integration by parts
give, for every smooth $(p-1,q)$-form $\beta$,
\[
\begin{aligned}
 \langle\partial\beta,u\rangle_{L^2(h)}
 &=\int_X h(\partial\beta,u)\,\mathrm{vol}
 =\int_X q_\Omega(\partial\beta,u)\,\mathrm{vol}\\
 &=(\sqrt{-1})^{\,p-q}(-1)^{(p+q)(p+q-1)/2}
   \int_X\partial\beta\wedge\bar u\wedge\Omega\\
 &=(\sqrt{-1})^{\,p-q}(-1)^{(p+q)(p+q-1)/2}
   \int_X\partial(\beta\wedge\bar u\wedge\Omega)=0,
\end{aligned}
\]
where we use $\partial\bar u=\overline{\bar\partial u}=0$ and
$\partial\Omega=0$.
By the definition of the formal adjoint, for every smooth
$(p-1,q)$-form $\beta$,
\[
 \langle\beta,\partial_h^*u\rangle_{L^2(h_0)}
 =\langle\partial\beta,u\rangle_{L^2(h)}=0.
\]
Hence $\partial_h^*u=0$ in the distributional sense.
Together with $\partial u=0$, it follows that
$(\partial\partial_h^*+\partial_h^*\partial)u=0$.
As in \cite[Proposition~3.2]{DN}, elliptic regularity for this
Laplacian implies that $u$ is smooth.

Finally, for every smooth closed $(n-p,n-q)$-form $\chi$, integration by parts gives
\[
 \int_X(u-f)\wedge\chi
 =-\int_Xv\wedge\mathrm{d}\mathrm{d}^c\chi=0.
\]
Poincar\'e duality therefore gives $[u]=[f]$.
For $[f]\ne0$, this smooth representative is nonzero, and (ii) gives
\[
 Q_{[\Omega]}([f],[f])=\int_Xq_\Omega(u,u)\,\mathrm{vol}>0.
\]
\end{proof}

\begin{theorem}\label{sp:thm:global}
Let $([\alpha],[\omega])$ be a cohomological supercritical LYZ pair on a compact
K\"ahler $n$-fold, $n\ge2$, and let
$\mathcal R:=\Rpair{[\alpha]}{[\omega]}$. For $p,q\ge0$ with
$1\le p+q\le n$ and $\min\{p,q\}\le1$, the following hold.
\begin{enumerate}[label=(\roman*)]
\item Multiplication by $c_{n-p-q}(\mathcal R)$ is an isomorphism
$H^{p,q}(X)\longrightarrow H^{n-q,n-p}(X)$.
\item The form $Q_{\mathcal R}^{p,q}$ is positive definite on
$\mathcal P^{p,q}(X,\mathcal R)$.
\item If $p,q\ge1$, there is a $Q_{\mathcal R}^{p,q}$-orthogonal decomposition
\[
 H^{p,q}(X)=\mathcal P^{p,q}(X,\mathcal R)\mathbin{\oplus^{\perp}}
 \mathsf L_{\mathcal R} H^{p-1,q-1}(X).
\]
If $p=0$ or $q=0$, the whole space is primitive.
\end{enumerate}
\end{theorem}
\begin{proof}
Choose representatives $\alpha,\omega$ of $[\alpha],[\omega]$, respectively,
such that $(\alpha,\omega)$ is a supercritical LYZ pair.
Then we set $m:=n-p-q+1$ and
$\Omega:=c_{m-1}(\alpha,\omega)$, so
$[\Omega]=c_{m-1}(\mathcal R)$.

By Proposition~\ref{sp:prop:equality}, we first apply Proposition~\ref{sp:prop:adjacent-global} to
$(\Psi,\Omega)=(\Omega\wedge\kappa,\Omega)$. It gives positivity of
$Q_{[\Omega]}$ on the kernel of $[\Omega\wedge\kappa]$.
If $[\Omega]\cup a=0$, then $a$ belongs to the kernel of $[\Omega\wedge\kappa]$ and
$Q_{[\Omega]}(a,a)=0$, so $a=0$.
The source and target of $[\Omega]\cup-$ have equal dimensions
by Hodge symmetry. Hence the map is an
isomorphism, proving (i). By Proposition~\ref{sp:prop:equality}, we may apply Proposition~\ref{sp:prop:adjacent-global} again to $(c_m(\alpha,\omega),c_{m-1}(\alpha,\omega))$, which proves (ii).

For (iii), when either index is zero, the primitive map has zero target, so we may assume $p,q\ge1$. By (i), the form $Q_{\mathcal R}^{p,q}$ is
nondegenerate and $\mathsf L_{\mathcal R}$ is defined.
By (ii), its restriction to $\mathcal P^{p,q}(X,\mathcal R)$ is positive
definite. Hence this subspace meets its orthogonal complement only at zero,
and nondegeneracy of $Q_{\mathcal R}^{p,q}$ gives
\[
 H^{p,q}(X)=\mathcal P^{p,q}(X,\mathcal R)\mathbin{\oplus^{\perp}}
 \bigl(\mathcal P^{p,q}(X,\mathcal R)\bigr)^{\perp_{Q_{\mathcal R}}}.
\]
It remains to identify the second summand with
$\operatorname{im}\mathsf L_{\mathcal R}$.
For $a\in H^{p,q}(X)$, the defining identity of $\mathsf L_{\mathcal R}$
shows that $a\in(\operatorname{im}\mathsf L_{\mathcal R})^{\perp_{Q_{\mathcal R}}}$
if and only if
\[
 \int_X a\cup c_m(\mathcal R)\cup\bar y=0
 \qquad\text{for every }y\in H^{p-1,q-1}(X).
\]
By Poincar\'e duality, this is equivalent to $c_m(\mathcal R)\cup a=0$.
Hence
\[
 \bigl(\operatorname{im}\mathsf L_{\mathcal R}\bigr)^{\perp_{Q_{\mathcal R}}}
 =\ker\bigl(c_m(\mathcal R)\cup-\bigr)
 =\mathcal P^{p,q}(X,\mathcal R).
\]
Taking orthogonal complements gives
$\bigl(\mathcal P^{p,q}(X,\mathcal R)\bigr)^{\perp_{Q_{\mathcal R}}}
=\operatorname{im}\mathsf L_{\mathcal R}$, and this proves (iii).
\end{proof}

\Needspace{6\baselineskip}
Combining Theorem~\ref{sp:thm:global} with the classical
middle-degree relations gives the following consequence in dimensions
three and four.

\begin{corollary}\label{cor:low-dimensional-package}
Let $([\alpha],[\omega])$ be a cohomological supercritical LYZ pair on a compact
K\"ahler $n$-fold, where $n=3$ or $4$.
We set $\mathcal R:=\Rpair{[\alpha]}{[\omega]}$. For $p,q\ge0$, the following hold.
\begin{enumerate}[label=(\roman*)]
\item For $1\le p+q\le n$, multiplication by $c_{n-p-q}(\mathcal R)$
is an isomorphism
\[
 H^{p,q}(X)\longrightarrow H^{n-q,n-p}(X).
\]
\item For $1\le p+q\le n$, the form $Q_{\mathcal R}^{p,q}$ is positive
definite on $\mathcal P^{p,q}(X,\mathcal R)$.
\item For $0\le p+q\le n$, there is a Lefschetz decomposition
\[
 H^{p,q}(X)=\bigoplus_{r=0}^{\min\{p,q\}}
 \mathsf L_{\mathcal R}^{\,r}\mathcal P^{p-r,q-r}(X,\mathcal R).
\]
\end{enumerate}
\end{corollary}
\begin{proof}
Theorem~\ref{sp:thm:global} proves (i) and (ii) except in bidegree
$(p,q)=(2,2)$ when $n=4$.
In this remaining case, we have $c_0(\mathcal R)=1$, and the class $c_1(\mathcal R)=3[\kappa]$ is K\"ahler.
Thus the Lefschetz map is the identity, while the primitive
space and pairing are those of the classical Hodge--Riemann
relations. Moreover, $\mathsf L_{\mathcal R}=c_1(\mathcal R)\cup-$, so classical
Lefschetz decomposition gives
\[
 H^{2,2}(X)=\mathcal P^{2,2}(X,\mathcal R)\mathbin{\oplus^{\perp}}
 \mathsf L_{\mathcal R} H^{1,1}(X).
\]
Substituting the decomposition of $H^{1,1}(X)$ from
Theorem~\ref{sp:thm:global}(iii) gives (iii), since multiplication
by $c_1(\mathcal R)$ is injective on $H^{1,1}(X)$ by classical Hard Lefschetz.
In degree zero, the decomposition is
$H^{0,0}(X)=\mathcal P^{0,0}(X,\mathcal R)$ by definition.
\end{proof}

These low-dimensional cases suggest that the same package should
hold in arbitrary dimension.

\begin{conjecture}\label{sp:conj:full-package}
Let $([\alpha],[\omega])$ be a cohomological supercritical LYZ pair on a compact
K\"ahler $n$-fold, $n\ge2$, and let
$\mathcal R:=\Rpair{[\alpha]}{[\omega]}$. For $p,q\ge0$, the following hold.
\begin{enumerate}[label=(\roman*)]
\item For $1\le p+q\le n$, multiplication by $c_{n-p-q}(\mathcal R)$
is an isomorphism
\[
 H^{p,q}(X)\longrightarrow H^{n-q,n-p}(X).
\]
\item For $1\le p+q\le n$, the form $Q_{\mathcal R}^{p,q}$ is positive
definite on $\mathcal P^{p,q}(X,\mathcal R)$.
\item For $0\le p+q\le n$, there is a Lefschetz decomposition
\[
 H^{p,q}(X)=\bigoplus_{r=0}^{\min\{p,q\}}
 \mathsf L_{\mathcal R}^{\,r}\mathcal P^{p-r,q-r}(X,\mathcal R).
\]
\end{enumerate}
\end{conjecture}

\begin{remark}
For every dimension $n\ge2$, we can prove Conjecture~\ref{sp:conj:full-package}
when the phase is sufficiently small. Work on the general case is
in progress.
\end{remark}

\appendix
\section{Toric data for the Fujino--Sato example}\label{app:FS-data}
We give the data used in Section~\ref{subsec:strict-inclusion}.
The variety and the ray numbering are those of
\cite[Example~4.2]{FS}.

\subsection*{Fan and divisor classes}
The primitive ray generators are
\[
\begin{alignedat}{2}
 v_1&=(1,0,0),       &\qquad v_5&=(1,-1,-2),\\
 v_2&=(0,1,0),       &\qquad v_6&=(1,0,-1),\\
 v_3&=(0,0,1),       &\qquad v_7&=(0,-1,-2),\\
 v_4&=(-1,-1,-1),    &\qquad v_8&=(0,0,-1).
\end{alignedat}
\]
We write $ijk$ for $\Cone(v_i,v_j,v_k)$. The maximal cones are
\[
\begin{array}{rrrrrr}
 123,&126,&134,&145,&156,&234,\\
 248,&256,&258,&457,&478,&578.
\end{array}
\]
They arise from the successive subdivisions in
\cite[Example~4.2]{FS}. The linear divisor relations are
\[
\begin{aligned}
 D_1-D_4+D_5+D_6&=0,\\
 D_2-D_4-D_5-D_7&=0,\\
 D_3-D_4-2D_5-D_6-2D_7-D_8&=0.
\end{aligned}
\]
Together with the Stanley--Reisner relations and
$\deg(D_iD_jD_k)=1$ for every maximal cone $ijk$, these determine
the Chow ring and its intersection numbers
\cite[Theorem~12.5.3 and~(12.5.9)]{CLS}.

The following table gives the nef basis and the expressions of the
pseudoeffective generators in that basis.
\begingroup
\renewcommand{\arraystretch}{1.35}
\begin{equation}\label{eq:FS-psef-basis}
\begin{array}{@{}l@{\qquad}l@{}}
 \text{Nef generators} & \text{Pseudoeffective generators}\\[2pt]
 N_1=D_3              & D_1=-2N_1+N_2\\
 N_2=D_1+2D_3         & D_5=N_1+N_2-N_3+N_4\\
 N_3=2D_1+4D_3+D_6    & D_6=-2N_2+N_3\\
 N_4=D_3+D_4          & D_7=-2N_4+N_5\\
 N_5=2D_3+2D_4+D_7    & D_8=N_3+N_4-2N_5
\end{array}
\end{equation}
\endgroup
The nef cone is computed from invariant curves using the wall
relations \cite[Proposition~6.4.4]{CLS}, the primitive relations in
\cite[Example~4.2]{FS}, and the nef-cone criterion
\cite[Theorem~6.3.12]{CLS}. By \cite[Lemma~15.1.8]{CLS}, the
pseudoeffective cone is generated by the invariant prime divisor classes.
Together these computations give \eqref{eq:FS-cones}.

\subsection*{Intersection numbers}
The Chow ring gives the following intersection matrix.

\begingroup
\renewcommand{\arraystretch}{1.35}
\setlength{\arraycolsep}{8pt}
\begin{equation}\label{eq:N4-intersection-matrix}
 \bigl(N_iN_jN_4\bigr)_{i,j}=
 \begin{pmatrix}
 2&6&12&4&8\\
 6&15&30&10&20\\
 12&30&60&20&40\\
 4&10&20&6&13\\
 8&20&40&13&26
 \end{pmatrix}.
\end{equation}
\endgroup

We obtain the following pairings of $\Gamma_*$ and $KM$ with the
pseudoeffective generators.
\begingroup
\renewcommand{\arraystretch}{1.4}
\setlength{\arraycolsep}{10pt}
\begin{equation}\label{eq:FS-exact-certificate}
\begin{array}{c@{\qquad}rrrrr}
          &D_1&D_5&D_6&D_7&D_8\\ \hline
 \Gamma_* &2&0&0&1&6\\
 KM       &19&12&2&1&19
\end{array}
\end{equation}
\endgroup

The next table lists the values of $M^2$ on the nef generators.

\begingroup
\renewcommand{\arraystretch}{1.4}
\setlength{\arraycolsep}{10pt}
\begin{equation}\label{eq:FS-M-squares}
\begin{array}{c@{\qquad}rrrrr}
      &N_1&N_2&N_3&N_4&N_5\\ \hline
 M^2  &25&41&86&0&1
\end{array}
\end{equation}
\endgroup

The remaining triple intersections are recorded below.
\begingroup
\renewcommand{\arraystretch}{1.4}
\setlength{\arraycolsep}{16pt}
\begin{equation}\label{eq:FS-triple-values}
\begin{array}{ccc}
 MK^2 & N_4^2K & D_8N_4^2\\ \hline
 959 &53 &0
\end{array}
\end{equation}
\endgroup

\end{document}